\documentclass[11pt]{amsart}

\usepackage[colorlinks=true, pdfstartview=FitV, linkcolor=blue, 
citecolor=blue]{hyperref}

\usepackage{amssymb,amsmath,accents}
\usepackage{bbm}
\usepackage{graphicx}
\usepackage{a4wide}

\DeclareMathAlphabet{\mymathbb}{U}{bbold}{m}{n}

\newtheorem{theorem}{Theorem}[section]
\newtheorem{prop}[theorem]{Proposition}
\newtheorem{lemma}[theorem]{Lemma}     
\newtheorem{fact}[theorem]{Fact}
\newtheorem{coro}[theorem]{Corollary}
\theoremstyle{definition}
\newtheorem{defin}[theorem]{Definition}
\newtheorem{example}[theorem]{Example}
\newtheorem{remark}[theorem]{Remark}

\newcommand{\ts}{\hspace{0.5pt}}
\newcommand{\nts}{\hspace{-0.5pt}}

\newcommand{\RR}{\mathbb{R}\ts}
\newcommand{\CC}{\mathbb{C}}
\newcommand{\KK}{\mathbb{K}}
\newcommand{\ZZ}{\mathbb{Z}}

\newcommand{\NN}{\mathbb{N}}

\newcommand{\SSS}{\ts\mathbb{S}}

\newcommand{\cA}{\mathcal{A}}

\newcommand{\cC}{\mathcal{C}}

\newcommand{\cE}{\mathcal{E}}

\newcommand{\cG}{\mathcal{G}}
\newcommand{\cM}{\mathcal{M}}

\newcommand{\cP}{\mathcal{P}}

\newcommand{\ee}{\ts\mathrm{e}}
\newcommand{\ii}{\mathrm{i}\ts}

\newcommand{\one}{\mymathbb{1}}
\newcommand{\nix}{\mymathbb{0}}
\newcommand{\Mat}{\mathrm{Mat}}
\newcommand{\Co}{\mathrm{cent}}
\newcommand{\alg}{\mathrm{alg}}
\newcommand{\tr}{\mathrm{tr}}
\newcommand{\degr}{\mathrm{deg}}

\newcommand{\diag}{\ts\ts\mathrm{diag}\ts}

\newcommand{\exend}{\hfill$\Diamond$}
\newcommand{\defeq}{\mathrel{\mathop:}=}
\newcommand{\eqdef}{=\mathrel{\mathop:}}

\newcommand{\myfrac}[2]{\frac{\raisebox{-2pt}{$#1$}}
      {\raisebox{0.5pt}{$#2$}}}

\begin{document}

\title[Notes on Markov embedding]
{Notes on Markov embedding}

\author{Michael Baake}
\address{Fakult\"at f\"ur Mathematik, Universit\"at Bielefeld, \newline
       \indent  Postfach 100131, 33501 Bielefeld, Germany}

\author{Jeremy Sumner}
\address{School of Mathematics and Physics, University of Tasmania,
    \newline \indent Private Bag 37, Hobart, TAS 7001, Australia}

\begin{abstract} 
  The representation problem of finite-dimensional Markov matrices in
  Markov semigroups is revisited, with emphasis on concrete criteria
  for matrix subclasses of theoretical or practical relevance, such as
  equal-input, circulant, symmetric or doubly stochastic matrices.
  Here, we pay special attention to various algebraic properties of
  the embedding problem, and discuss the connection with the
  centraliser of a Markov matrix.
\end{abstract}


\maketitle

\section{Introduction}

A \emph{stochastic} or \emph{Markov matrix} $M$ is a matrix with
non-negative entries and row sums~$1$, which is the convention we use
here. We call $M$ \emph{positive} when \emph{all} its entries are
positive, and the zero matrix is singled out from the non-negative 
matrices by calling it
\emph{trivial}.  A Markov \emph{generator} $Q$, also known as a
\emph{rate matrix}, has non-negative entries off the diagonal and row
sums $0$, and $\{ \ee^{t\ts Q} : t \geqslant 0\}$ is the homogeneous
Markov \emph{semigroup} (or monoid, to be more precise) generated by
$Q$; see \cite{Chung} for general background.  It is an old and still
only partially resolved question which Markov matrices $M$ are
\emph{embeddable}, meaning that they appear in a Markov semigroup. For
a given $M$, this is clearly equivalent to the existence of a rate
matrix $Q$ such that $M=\ee^Q$.

While the traditional focus was on irreducible Markov matrices, recent
progress on models of biological evolution has put new emphasis also
on reducible and even on absorbing Markov chains, and on submodels
with additional algebraic structure; see \cite{fast,Servet,FSJW,SFSJ}
for recent investigations.  Moreover, an answer to the embedding
problem gives additional insight into Markov semigroups that are being used
in many places; see \cite{BN} and references therein for recent
examples.  This motivates us to revisit the embedding problem with an
eye on more specific \emph{classes} of Markov matrices. Here, 
we only consider
the finite-dimensional case, which gives nice relations to rather
helpful algebraic structures.  Nevertheless, it is inevitable to get
in contact with the topologically complicated structure of the
boundary of the set of embeddable matrices (as detailed in Kingman's
influential paper \cite{King} on the subject), which can be seen as
one of the reasons why the problem has not yet received a complete
solution.

An efficient starting point to the embedding problem is the paper by
Davies \cite{Davies}, who collected a good number of known results and
examples in one place, with a useful list of references. We will refer
to it frequently. Also, \cite{Joh1} gives an insightful overview of
many of the early results. Some care is required in identifying
the precise conditions in informal statements, which are often
implicit, such as irreducibility or a positive determinant. One
should note that many of the early results give abstract
characterisations that are of limited use in practice.  Our goal here
is to revisit the embedding problem from a slightly different
perspective, where we treat it in a more concrete fashion for
various classes of matrices that are natural from an algebraic point
of view or that frequently show up in applications.\smallskip

The paper is organised as follows. In Section~\ref{sec:general}, we
set the scene, recall some of the known results, and formulate various
algebraic and asymptotic properties for later use. The classic case of
two-dimensional ($d=2$) Markov matrices is reviewed in
Section~\ref{sec:two}, where we select calculations and proofs with an
eye to our later needs in higher dimensions.  One generalisation, the
class of \emph{equal-input} matrices, is then treated in
Section~\ref{sec:equal-input}, where an interesting dichotomy between
even and odd dimensions shows up. Section~\ref{sec:circulant} presents
results for the class of \emph{circulant} matrices, which have a rich
theory of their own.  Finally, we discuss various more specialised
classes of Markov matrices for $d=3$ in Section~\ref{sec:three}, once
again aiming at more concrete criteria, and close with a brief outlook
in Section~\ref{sec:outlook}.

\section{General setting and results}\label{sec:general}

Let us begin by recalling some necessary (but generally far from
sufficient) conditions for embeddability. For convenience, we give
brief hints on the proofs or references. We use $\sigma (A)$ to denote
the \emph{spectrum} of a matrix $A$, usually including
multiplicities. When the latter are not important, we simply consider
$\sigma (A)$ as a set.  Let us also recall that the $d$-dimensional
Markov matrices form a closed, convex subset
$\cM_d \subset \Mat (d, \RR)$, which has locally constant, topological
dimension $d \ts (d\! - \! 1)$, where we generally assume
$d\geqslant 2$ to avoid trivialities.\footnote{Here and below, we use
  $\Mat (d,\KK)$ to denote the ring of $d \!\times\!\nts d$-matrices
  over the field $\KK$.}  Clearly, $\cM_d$ is a monoid with respect to
matrix multiplication, where the extremal elements are the stochastic
$\{ 0, 1 \}\ts$-matrices \cite{Joh2}. Also, if $M$ is Markov, one has
$\det (M) \leqslant 1$, with equality if and only if $M$ is a
permutation matrix for an even permutation. The last property follows
from general Perron--Frobenius theory applied to $M$ together with the
fact that Markov matrices with determinant $1$ are diagonalisable
\cite[Sec.~13.6]{Gant}.

\begin{prop}\label{prop:necessary}
  If a Markov matrix\/ $M$ is embeddable, so\/ $M = \ee^Q$ 
  with\/ $Q$ a Markov generator, $M$ satisfies the following
  properties.
\begin{enumerate}\itemsep=2pt  
  \item\label{part:one} The spectra are related by\/ 
      $\sigma (M) = \ee^{\ts \sigma (Q)}$.
  \item\label{part:two} One has\/ $\ts 0< \det (M) \leqslant 1$,
      so\/ $0\notin \sigma (M)$, and\/ $\det (M)=1$ only for\/
      $M=\one$.
  \item\label{part:Elving} If\/ $\lambda \in \sigma (M)$ with\/ 
      $\lambda \ne 1$, then\/ $\lvert \lambda \rvert < 1$.
  \item\label{part:neg} Each real\/ $\lambda \in \sigma (M)$ 
      with\/ $\lambda < 0$ must have even algebraic multiplicity.    
  \item\label{part:four} $M$ is either reducible, or positive and
      thus also primitive.
  \item If\/ $M_{ij}>0$ and\/ $M_{jk}>0$, then also\/ $M_{ik}>0$.
  \end{enumerate}  
\end{prop}  

\begin{proof}
  Claim (1) is clear from the spectral mapping theorem, while (2)
  follows from the identity $\det (\ee^Q) = \ee^{\tr (Q)}$ with
  $\tr (Q) \leqslant 0$, where $\tr (Q) = 0$ means $Q=\nix$.
  
  Property (3) is Elving's theorem \cite{Elving}, compare
  \cite[Prop.~8]{Davies}, while Property (4) is shown in
  \mbox{\cite[Prop.~2]{Davies}}.  The remaining claims follow from
  standard results on the structure of $\ee^{t \ts Q}$ for
  $t\geqslant 0$; see \cite[Thm.~3.2.1]{Norris} for a general
  statement.
\end{proof}

Let us note in passing that the difficulty of solving $M=\ee^Q$ for
$Q$ consists in the existence of a logarithm of $M$ that has the
positivity properties required for a generator, where the latter
(known as the Metzler property) is the harder constraint by far when
$d>2$.

\begin{example}\label{ex:trivex}
  While $\one = \ee^{\nix}$ is trivially embeddable, any irreducible
  Markov matrix that is embeddable must actually be primitive and
   positive. For instance,
  $\left( \begin{smallmatrix} 0 & 1 \\ 1 & 0
    \end{smallmatrix} \right)$ is irreducible, but not primitive,
  while
  $\left( \begin{smallmatrix} 1-a & a \\ 1 & 0 \end{smallmatrix}
  \right)$ with $a\in (0,1)$ is primitive, but not positive, so
  neither of these matrices is embeddable; compare
  Example~\ref{ex:paradox} below for more.

  Also,
  $M=\left( \begin{smallmatrix} 1-a & a \\ a & 1-a \end{smallmatrix}
  \right)$, which has spectrum $\sigma (M) = \{ 1, 1 - 2a\}$, cannot
  be embeddable for $a\in \bigl[ \frac{1}{2} ,1 \bigr]$, as this
  violates Proposition~\ref{prop:necessary}{\ts}(\ref{part:two}), as
  well as \ref{prop:necessary}{\ts}(\ref{part:neg}) when
  $a>\frac{1}{2}$.  \exend
\end{example}

Homogeneous Markov semigroups possess a well-known asymptotic
property, which we recall here for convenience and later use; compare
\cite[Thms.~12.25 and 12.26]{Kalle} for closely related results.
Also, some aspects of Proposition~\ref{prop:necessary} may become more
transparent this way.
  
\begin{prop}\label{prop:G-asymp}
  Every finite-dimensional Markov generator\/ $Q$ has the following
  properties.
\begin{enumerate}\itemsep=2pt
\item If\/ $\lambda \in \sigma (Q)$, one either has\/ $\lambda = 0$
  or\/ $\mathrm{Re} (\lambda) < 0$. Eigenvalues of\/ $Q$ are either
  real or occur in complex conjugate pairs.
\item The minimal polynomial of\/ $Q$ is of the form\/ $z \, q(z)$
  with\/ $q(0) \ne 0$, which is to say that the algebraic and the
  geometric multiplicity of\/ $\lambda=0$ coincide.
\item If\/ $ M(t) \defeq \ee^{t Q}$, the limit\/
  $M^{}_{\infty} = \lim^{}_{t\to\infty} M (t)$ exists and is a Markov
  matrix with\/ $M^{2}_{\infty} = M^{}_{\infty}$. As such, it is
  diagonalisable, with\/
  $1 \in \sigma (M^{}_{\infty}) \subseteq \{ 0,1\}$.
\end{enumerate}   
\end{prop}  

\begin{proof}
  While claim (1) follows from
  Proposition~\ref{prop:necessary}{\ts}(\ref{part:two}) via the
  spectral mapping theorem, we prefer to give an independent argument
  with some additional insight.  Define the number
  $\mu = \min \{ z \geqslant 0 : Q + z \ts \one \ts \text{ is a
    non-negative matrix} \}$ and set $R = Q + \mu \ts \one$, with
  elements $r^{}_{ij} \geqslant 0$ for all
  $1 \leqslant i,j \leqslant d$, and $\sum_j r^{}_{ij} = \mu$ for all
  $i$ by construction. The Gershgorin circles of $R$ are
  $G^{}_i = B^{}_{\mu - r_{ii}} (r^{}_{ii})$, one of which must be
  $B_{\mu} (0)$, where $B_{\rho} (x)$ is the closed disk of radius
  $\rho$ around $x$. Clearly, we then have
  $\bigcup_i G_i = B_{\mu} (0)$, and $\sigma (R) \subset B_{\mu} (0)$
  by Gershgorin's theorem \cite[Thm.~14.6]{Gant}, so
  $\sigma (Q) \subset B_{\mu} (-\mu)$. Since $Q$ is a real matrix,
  this gives the first claim.

  We know that $0 \in \sigma (Q)$, as $Q$ has zero row sums.  To show
  claim (2), assume to the contrary that the geometric multiplicity of $0$ is
  smaller than the algebraic one. Then, there exists a
  row vector\footnote{Though an equivalent argument can
    be given with a column instead of a row vector, this version has the
    slight advantage that we can directly use the row sum
    normalisation of $M$.}  $u \ne 0$ such that $uQ^2=0$ but
  $uQ=v\ne 0$. For $M\defeq \ee^Q$,  this 
  implies $uM = u+v$ and thus $uM^n = u+ n \ts v$ for $n\in\NN$ 
  by induction.  If $\|.\|^{}_{1}$ denotes the $1$-norm for row
  vectors, the matching matrix norm is the row sum norm, with
  $\| M \|^{}_{1} = 1$ because $M$ is a Markov matrix. 
  Consequently, we get
  $\| u M^n \|^{}_{1} \leqslant \| u \|^{}_{1} \, \| M \|^{n}_{1} =
  \|u \|^{}_{1} $, which is bounded. In contrast, we have
  $\| u + n \ts v \|^{}_{1} \geqslant n \ts \| v \|^{}_{1} - \| u
  \|^{}_{1}$, which is unbounded due to $\| v\|^{}_{1} >0$ and thus
  gives a contradiction. Therefore, the vector $u \in \mathrm{ker}(Q^2)
  \setminus \mathrm{ker} (Q)$ cannot exist.

  Claim (3) is a simple consequence of Claims (1) and (2) in
  conjunction with the observation that the Markov property is
  preserved under taking limits, because $\cM_d$ is closed in
  $\Mat (d,\RR)$. The projector property follows from
\[
   M^{2}_{\infty} \, = \, \Bigl(\ts \lim_{t\to\infty} \ee^{t Q}
   \Bigr)^2  = \, \lim_{t\to\infty} \ee^{2 t Q} \, = \, M^{}_{\infty}
   \ts ,
\]
  which also implies the claim on the eigenvalues.
\end{proof}

One immediate consequence is the following. If $\one\ne M=\ee^Q$ is
embeddable, the Markov semigroup $\{ \ee^{t Q} : t \geqslant 0 \}$
defines a path (in $\cM_d$) of embeddable matrices from $\one$
(included) via $M$ to $M_{\infty}$ (not included). Here, since
$Q\ne\nix$ and thus $\tr (Q) < 0$, one has $\det (M_{\infty})=0$, and
$M_{\infty}$ itself is not embeddable.  In this way, no embeddable
Markov matrix is isolated, and any two embeddable matrices of the same
dimension are pathwise connected.

When $M$ is a general Markov matrix (hence not necessarily
embeddable), it is also true that $1$ is an eigenvalue with equal
algebraic and geometric multiplicity \cite[Thm.~13.10]{Gant}, which
can be seen as a consequence of the normal form of non-negative
matrices in \cite[Eq.~(13.70)]{Gant}.  However, $M^n$ need not
converge as $n\to\infty$, because $M$ can be irreducible without being
primitive; compare \cite[Thms.~8.18 and 8.22]{Kalle}. This cannot
occur for embeddable matrices, in line with Elving's theorem, which is
Proposition~\ref{prop:necessary}{\ts}(\ref{part:Elving}) above.

\begin{coro}\label{coro:M-asymp}
  If\/ $M$ is an embeddable Markov matrix,
  $M^{}_{\infty} = \lim_{n\to\infty} M^n$ exists and is again a Markov
  matrix.  Moreover, $R=M^{}_{\infty}-\one$ is a generator that
  satisfies\/ $R^2 = -R$. As such, it is diagonalisable, with\/
  $0 \in \sigma (R) \subseteq \{ -1, 0 \}$.
  
  More generally, with\/
  $\SSS\defeq \{ z \in \CC : \lvert z \rvert = 1 \}$ denoting the unit
  circle, the same conclusions hold for any Markov matrix\/ $M$ with\/
  $\sigma (M) \cap \SSS = \{ 1 \}$.
\end{coro}

\begin{proof}
  Though the claims for embeddable matrices follow from the more
  general ones via Elving's theorem, we give a simple independent
  argument.  When $M=\ee^Q$, one has $M^n = \ee^{nQ}$, and the first
  claim follows from \mbox{Proposition~\ref{prop:G-asymp}{\ts}(3)}.
  The generator property of $R$ is clear, while
  $M^{2}_{\infty} = M^{}_{\infty}$ implies the relation for $R$ as
  well as the property of the spectrum.
  
  For the general claim, recall the comment made before the
  corollary. It implies that such an $M$ has minimal polynomial
  $(z-1) \ts\ts p(z)$ with $p(1)\ne 0$. All roots of $p$ have modulus
  $< 1$, so convergence follows from a standard Jordan normal form
  argument.
\end{proof}

Let $E_d$ denote the set of $d$-dimensional Markov matrices that are
embeddable, and let $\cE_d \defeq \langle E_d \rangle$ be the
(multiplicative) semigroup generated by it.  In view of the dichotomy
in Proposition~\ref{prop:necessary}{\ts}\eqref{part:four}, we also
introduce $\cE^+_d \eqdef \{ M\in \cE_d : M \text{ is positive} \}$.

\begin{fact}\label{fact:ideal} 
  $\cE_d$ is a monoid, while\/ $\cE^+_d$ is a semigroup and a
  two-sided ideal in\/ $\cE_d$.
\end{fact}

\begin{proof}
  Since $\one$ is embeddable, $\cE_d$ is a semigroup with unit, and
  the first claim is clear. For the second, observe that the product
  of positive Markov matrices is positive, which implies the semigroup
  property, and that the product of a positive Markov matrix with any
  Markov matrix, in either order, is again a positive Markov matrix,
  which implies $\cE^+_d$ to be a two-sided ideal in $\cE_d$.
\end{proof}

The set $E_d$ is relatively closed \cite[Prop.~3]{King} within
$\cM^{>}_{d} \defeq \{ M \in \cM^{}_{d} : \det (M) > 0 \}$, with the
same topological dimension. In particular, one has the
interior-closure inclusions
\begin{equation}\label{eq:include}
    \overline{E^{}_d}^{\,\circ} \, \subset \, E^{}_{d}
    \, \subset \, \overline{E^{\circ}_{d}} \ts ,
\end{equation}
as follows from \cite[Prop.~4]{King}.  Moreover, $E_d$ contains
various important relatively open subsets, as detailed in
\cite[Lemma~3 and Thm.~7]{Davies}. The proofs use some standard
deformation arguments, which can easily be extended to give the
following result.

\begin{fact}\label{fact:dense}
  The set\/ $E_d$ of embeddable, $d$-dimensional Markov matrices
  contains the following dense and relatively open subsets, namely
\begin{enumerate}\itemsep=2pt
\item the embeddable matrices without eigenvalues
    in\/ $\{ x\in\RR : x \leqslant 0 \}$;
\item the embeddable matrices with distinct eigenvalues;
\item the embeddable matrices with Abelian centraliser in\/
  $\Mat (d,\RR)$, as characterised in more detail below in
  Fact~\textnormal{\ref{fact:comm}}.
\end{enumerate}
 Moreover, each of these subsets has full measure in\/ $E_d$.
 \qed
\end{fact}

Let us take a different perspective on the general embedding problem,
which will provide a slightly simpler setting where some concrete
answers are possible. Assume that we consider a (finite or infinite)
set of rate matrices which generate a closed matrix algebra $\cG$ over
$\RR$ with respect to matrix addition and multiplication. For
$G\in\cG$ and $t\in\RR$, the series
\[
    \ee^{t G} \: = \: \one \ts + \sum_{m\geqslant 1}
    \myfrac{t^m}{m!} \, G^m 
\]
converges compactly in any matrix norm, where $G^m \in \cG$ for all
$m\in\NN$, thus $A\defeq \ee^G - \one \in \cG$ as well. Together with
Proposition~\ref{prop:G-asymp}{\ts}(\ref{part:one}) and its proof, one
has the following connection.

\begin{fact}\label{fact:connection}
  For any\/ $M \in \cM_d$, the matrix\/ $A\defeq M \! - \! \one$ is a
  Markov generator that commutes with\/ $M$ and satisfies\/
  $\sigma (A) \subset B^{}_{\mu} (-\mu)$ for some\/
  $0\leqslant \mu \leqslant 1$.  If\/ $M$ is embeddable, with\/
  $M = \ee^Q$ say, one has\/ $[A,Q \ts ]=\nix$.  \qed
\end{fact}

When $d\geqslant 2$, it is worth mentioning that, due to the
Cayley--Hamilton theorem and the structure of the exponential series,
one also has a representation
\begin{equation}\label{eq:CH}
    \ee^{t Q}  \, = \, \one \, + \sum_{\ell=1}^{d-1}
    F^{}_{\ell} (t) \, Q^{\ell} ,
\end{equation}
where each $F^{}_{\ell}$ is a power series with infinite convergence
radius and $F^{}_{\ell} (0) = 0$. When $Q$ is a rate matrix, this
structure suggests to consider the (non-unital) matrix algebra
\begin{equation}\label{eq:def:alg}
    \alg (Q) \, \defeq \, \big\langle Q, Q^2, \ldots , Q^{d-1} 
    \big\rangle^{}_{\RR} \ts ,
\end{equation}
which will become relevant shortly. Note that $\alg (Q)$ is a
subalgebra of the matrix algebra $\cA_0$ of all
real matrices with vanishing row sums, 
\begin{equation}\label{eq:A0}
    \cA_0 \, \defeq \, \big\{ A \in \Mat (d,\RR) : 
    \textstyle{\sum_{j=1}^{d}}  A_{ij} = 0 
    \text{ for all $1\leqslant i \leqslant d$} \big\},
\end{equation}
which does not contain $\one$ and is thus non-unital, too. 
Clearly, $\alg (Q)$ has dimension
$d \nts\nts - \! 1$ or smaller, where the latter case occurs if $Q$
has a minimal polynomial of degree less than $d$. Also, since both
$\cA_0$ and $\alg (Q)$ are closed subspaces of $\Mat (d, \RR)$ when
viewed as real vector spaces, they are complete with respect to any
chosen matrix norm on $\Mat (d,\RR)$, which is another way to see that
$\ee^{t Q} - \one$, for any $t\geqslant 0$, lies in $\alg (Q)$.

\begin{lemma}\label{lem:QandA}
  Let\/ $Q$ be a Markov generator, with\/ $M(t) = \ee^{t\ts Q}$
  defining the corresponding semigroup. Set\/ $M=M(1)=\ee^Q$ together
  with\/ $A=M-\one$ and\/ $R = M_{\infty}-\one$ as before, where\/ $R$
  is well defined by Corollary~\textnormal{\ref{coro:M-asymp}}.  Then,
  $R \in \alg (Q) \cap \ts \alg (A)$.
\end{lemma} 

\begin{proof}
  Clearly, both $A$ and $R$ are generators. Since
  $\ee^{t\ts Q}-\one \in \alg (Q)$ for any $t\geqslant 0$, its limit
  as $t\to\infty$, which is $R$, lies in $\alg (Q)$ as well, because
  the latter is closed.
   
   Now, observe
\[
     R \, = \lim_{n\to\infty} \bigl( \ee^{nQ} - \one \bigr)
     \, = \lim_{n\to\infty}  \bigl( (A + \one)^n  - \one \bigr),
\]
where
$(A+\one)^n - \one = \sum_{m=1}^{n} \binom{n}{m}  A^{m} \in
\alg (A)$. Since $\alg (A)$ is closed, we also have $R\in \alg (A)$,
and our claim follows.
\end{proof}

Next, we state a simple special case of $\alg (Q)$ for later use,
which harvests the fact that the only nilpotent Markov generator is
$Q=\nix.$

\begin{fact}\label{fact:one-dim}
  Let\/ $Q$ be a Markov generator, with\/ $d\geqslant 2$. Then, the
  following properties are equivalent.
\begin{enumerate}\itemsep=2pt
\item The minimal polynomial of\/ $Q$ is\/ $z (z+c)$ with\/ $c>0$.
\item $\ts\ts Q$ is diagonalisable, with eigenvalues\/ $0$ and\/ $-c$
  for some\/ $c>0$.
\item The matrix algebra\/ $\alg (Q)$ is one-dimensional.
\end{enumerate}    
\end{fact}

\begin{proof}
  (1) $\Longleftrightarrow$ (2) is standard, while (1) $\Rightarrow$
  (3) follows from $Q\ne \nix$ and $Q^2 = -c \ts\ts Q$.  Now, let us
  assume (3), where $\dim (\alg (Q)) =1$ implies $Q^2 = - c\ts \ts Q$
  for some $c \in \RR$.  Via
  Proposition~\mbox{\ref{prop:G-asymp}{\ts}(1)}, we see that
  $Q^2=\nix$ implies $Q=\nix$, which would contradict
  $\dim(\alg(Q))=1$. So, we have $c\ne 0$ and
  $Q (Q + c \ts \one) = \nix$, which tells us that $z (z+c)$ is the
  minimal polynomial of $Q$. As $-c \in \RR$ is then an eigenvalue of
  $Q$, another application of Proposition~\ref{prop:G-asymp}{\ts}(1)
  implies $c>0$, so (1) holds, and we are done.
\end{proof}

Let us return to Fact~\ref{fact:connection}. Viewed differently, the
generator of an embedding has to satisfy a commutativity
condition. Generically, this will mean that $Q$ is of the same `type'
as $A$, and it is thus an interesting subproblem of the general
embedding question to look at embeddability under a natural constraint
on $Q$, as discussed in \cite{Jeremy,SFSJ,FSJW}.  This can be of
practical relevance, for instance in biological applications; see
\cite[Ch.~7]{Steel} and references therein.

Since commutativity of matrices will be important in what follows, let
us recall a classic result from algebra
\cite[Chs.~III.15{\ts}--18]{Jacobson}, which is due to Frobenius.
Here, we specialise it to the case of real matrices; see also
\cite[Thm.~5.15 and Cor.~5.16]{AW} for another presentation.  For its
formulation, we employ the standard eigenspace notation
$V^{}_{\lambda} \defeq \{ x \in \CC^d : Bx = \lambda x \}$ for
$\lambda \in \CC$ and a given matrix
$B \in \Mat (d,\RR) \subset \Mat (d,\CC)$.

\begin{fact}\label{fact:comm}
  For\/ $B \in \Mat (d,\RR)$, the following properties are
  equivalent and generic.
\begin{enumerate}\itemsep=2pt
 \item The characteristic polynomial of $B$ is also its minimal
   polynomial.
 \item The relation\/ $\dim (V^{}_{\lambda}) =1$ 
   holds for every\/ $\lambda \in \sigma (B)$.    
 \item $B$ is\/ \emph{cyclic:} $\{ u, Bu, \ldots , B^{d-1} u \}$
   is a basis of\/ $\RR^d$ for some\/ $u\in\RR^d$.
 \item The matrix ring\/
   $\ts \Co (B) \defeq \{ C \in \Mat (d,\RR) : [B,C]=\nix \}$ is
   Abelian.
 \item One has\/ $\Co (B) = \RR [B]$, the ring of polynomials in 
   $B$ with coefficients in\/ $\RR$.  \qed
\end{enumerate}
\end{fact}

Below, we will call matrices of this kind \emph{cyclic}.  In
particular, Fact~\ref{fact:comm} applies to all matrices with
\emph{simple spectrum}, which means distinct eigenvalues. In the case
of degeneracies, the existence of additional elements in $\Co (B)$,
which is known as the \emph{centraliser} (or the \emph{commutant}) of
$B$, can be seen from a repetition of one of its Jordan blocks. In
this context, the last claim of Fact~\ref{fact:connection} means that
$A=M - \one$ with $M=\ee^Q$ implies $A\in\Co (Q)$ as well as
$Q\in \Co(A)$. This has the following consequence.

\begin{lemma}\label{lem:type}
  If\/ $M = \ee^Q$ is an embeddable Markov matrix that is
  cyclic, one has\/ $Q \in \alg (A)$ with\/ $A=M - \one$.
\end{lemma}

\begin{proof}
  For the matrix $A$, we are in the situation of Fact~\ref{fact:comm}.
  Then, by Fact~\ref{fact:connection}, the rate matrix $Q$ must be an
  element of $\RR [A]\cap \cA_0$, which means $Q \in \alg (A)$.
\end{proof}

The significance of this result emerges from the following
approximation concept.

\begin{defin}\label{def:stable}
  Consider an embeddable matrix $M\in\cM_d$, and set $A=M-\one$.  A
  property of $M$, say Property $\mathsf{S}$, is called \emph{stable}
  if it satisfies the following three conditions:
\begin{enumerate}\itemsep=2pt
\item Property\/ $\mathsf{S}$ has a well-defined analogue for
  generators, which holds for $A$ and for all generators in $\alg (A)$;
\item Property\/ $\mathsf{S}$ is preserved under taking limits within
  $\Mat (d,\RR)$;
\item $M$ can arbitrarily well be approximated (in some matrix norm)
  by other embeddable Markov matrices with Property $\mathsf{S}$ that
  are also cyclic (in the sense of Fact~\textnormal{\ref{fact:comm}}).
\end{enumerate}      
When only (1) and (2) are satisfied, we call the property
\emph{semi-stable}.
\end{defin}

As an example, think of a symmetric Markov matrix, so $M^T = M$, where
$A=M-\one$ and all element of $\alg (A)$ are then symmetric as well.
Clearly, the limit of a convergent sequence of symmetric matrices is
again symmetric.  So, being symmetric is a semi-stable property. It is
not a stable one though, since symmetric matrices have real
eigenvalues, and the existence of negative eigenvalues with even
multiplicity can cause a problem, as we shall see.
It is thus natural to also assume that $M$ has non-negative
spectrum. Then, it is the limit of sufficiently many convergent
sequences of symmetric matrices, in the sense that every neighbourhood
of such a matrix with non-negative eigenvalues also contains further
symmetric ones that are cyclic. This can be shown by a deformation
argument, and still remains true within $E_d$; see
Remark~\ref{rem:I-mat} below.  In contrast, the property of having
simple spectrum already fails to be semi-stable, as does having
geometric multiplicity $1$ for all eigenvalues.

Let us briefly comment on an algebraic consequence of
Definition~\ref{def:stable}. When Property $\mathsf{S}$ is a
\emph{linear} property (such as being symmetric), semi-stability means
that the $\mathsf{S}$-generators span a \emph{Jordan algebra}. This
follows from the simple observation that $A^2$, $B^2$ and $(A+B)^2$
having Property $\mathsf{S}$ then implies that $(AB+BA)$ has Property
$\mathsf{S}$ as well.

\begin{prop}\label{prop:extra}
  Let\/ $M$ be an embeddable Markov matrix.  If $\ts M$ has a stable
  property, say Property\/ $\mathsf{S}$, $M$ can be embedded as\/
  $M=\ee^Q$ with a generator\/ $Q$ having Property\/ $\mathsf{S}$. If,
  in addition, $M$ is cyclic, every generator\/ $Q$ with\/ $\ee^Q=M$
  must be an\/ $\mathsf{S}$-generator.
\end{prop}

\begin{proof}
  If $M$ itself is cyclic, its centraliser is Abelian, and we are in
  the situation of Lemma~\ref{lem:type}, so the first claim is clear
  in this case. Otherwise, Condition (3) of
  Definition~\ref{def:stable} implies that there is a sequence of
  embeddable, cyclic Markov matrices \emph{with}
  Property~$\mathsf{S}$, say $(M_j)^{}_{j\in\NN}$, such that
  $\lim_{j\to\infty} M_j = M$.
   
  Now, for all $j\in\NN$, we have $M_j = \ee^{Q_j}$ with
  $Q_j \in \alg (M_j - \one)$ by Lemma~\ref{lem:type}, where all $Q_j$
  have Property $\mathsf{S}$ due to Condition (1) of
  Definition~\ref{def:stable}. The set of solutions of
  $M_j = \ee^{Q'}$ with $Q'$ a Markov generator is discrete and
  finite, see Fact~\ref{fact:discrete} below, and the exponential map
  is locally a homeomorphism. Consequently, there is a subsequence
  $(j_m)^{}_{m\in\NN}$ of indices such that
  $(Q_{j^{}_m})^{}_{m\in\NN}$ converges, where
  $Q = \lim_{m\to\infty} Q_{j^{}_m}$ is a generator with property
  $\mathsf{S}$. By construction and Condition (2) of
  Definition~\ref{def:stable}, we get $M = \ee^Q$ as claimed.
  
  The second claim is now another consequence of 
  Lemma~\ref{lem:type} and Fact~\ref{fact:comm}.
\end{proof}

Notice that, by \cite[Prop.~4]{King}, the relative interior of $E_d$
as a subset of $\cM^{>}_{d}$ is non-empty, which means that the
topological dimension of $E_d$ is the maximal one, namely
$d\ts (d\! - \!1)$. This will actually help to identify when
Condition (3) of Definition~\ref{def:stable} is satisfied.

Let us also note the following variant of \cite[Cor.~10]{Davies}. The
proof is essentially the same, with the only difference coming from
the more general commutativity result in Fact~\ref{fact:comm}, that
is, using cyclic rather than simple matrices, and thus extends
\cite[Thm.~1.1]{Joh2}.

\begin{fact}\label{fact:discrete}
  Let\/ $M$ be a Markov matrix that is cyclic.  Then, the solutions
  of\/ $M = \ee^Q$ form a discrete set in $\cA_0$, and these solutions
  commute with one another and with\/ $M$. Moreover, only a finite
  number of the solutions can be Markov generators.  \qed
\end{fact}

Finally, let us recall a diagonalisability result that will come in
handy later.

\begin{fact}\label{fact:diag}
  Let\/ $B\in\Mat (d,\CC)$. If\/ $M=\ee^B$ is diagonalisable, then so
  is\/ $B$.
\end{fact}

\begin{proof}
  Via the Jordan--Chevalley decomposition, $B$ can be written as
  $B=D+N$ with $D$ diagonalisable, $N$ nilpotent, and $[D,N]=\nix$.
  In particular, the minimal polynomial of $N$ is $z^k$, for some
  $1\leqslant k \leqslant d$.  Then, $\ee^B = \ee^D \ee^N$ where
  $\ee^D$ is still diagonalisable, while $\ee^N = \one + N'$, with
  $N'=\sum_{m=1}^{k-1} \frac{1}{m!} N^m$ again being nilpotent.

  Since $M$ is diagonalisable, and $\ee^N = \ee^{-D} M$ with
  $[D,M]=\nix$, the matrix $\ee^N$ is diagonalisable as well, and we
  must have $N'=\nix$, as this is the only diagonalisable nilpotent
  matrix.  But this implies $k=1$, as any larger $k$ would give a
  contradiction to $z^k$ being the minimal polynomial of $N$. This
  means $N=\nix$ together with $B$ diagonalisable.
\end{proof}

\begin{remark}
   In terms of any explicit Jordan form of $B$ versus $\ee^B$,
   Fact~\ref{fact:diag} can also be seen as a consequence of
   \cite[Thm.~1.36{\ts}(a)]{Higham}, since the derivative of
   $\ee^x$ has no zero in $\CC$.  \exend
\end{remark}

We are now set to embark on the Markov embedding problem.

\section{Two-dimensional Markov matrices}\label{sec:two}

The situation for $d=2$ is particularly simple, because both
eigenvalues are real and the determinant condition from
Proposition~\ref{prop:necessary}{\ts}\eqref{part:two} 
actually is sufficient to
give a complete characterisation of $E^{}_2$, via
$E^{}_2 = \cM^{>}_{2}$. We recall this well-known result and give a
short explicit proof, both for the convenience of the reader and for
later use; for background on the underlying calculations around matrix
functions, we refer to \cite{Higham}.

\begin{theorem}[{Kendall; see \cite{King}}]\label{thm:Kendall}
  For a Markov matrix\/
  $M=\left( \begin{smallmatrix} 1-a & a \\ b & 1-b \end{smallmatrix}
  \right)$
  with\/ $a, b \in [0,1]$, the following statements are equivalent.
\begin{enumerate}\itemsep=2pt
\item $M$ is embeddable.
\item
  $0 < \det (M) = 1 - a - b \leqslant 1$.
\item 
  $1 < \tr (M) \leqslant 2$, which means\/
  $0 \leqslant a + b < 1$.
\item $M$ has positive spectrum, that is, positive eigenvalues only.  
\end{enumerate}
\end{theorem}

\begin{proof}
  All Markov matrices for $d=2$ are covered by the parametrisation
  chosen. The equivalence of (2), (3) and (4) is elementary, where (2)
  is a necessary condition for embeddability by
  Proposition~\ref{prop:necessary}{\ts}(\ref{part:two}), so (1)
  $\Rightarrow$ (2) $\Longleftrightarrow$ (3) $\Longleftrightarrow$
  (4) is clear.
  
  Now, $\det (M) > 0$ means $a+b<1$.  The embedding for the trivial
  case $a = b = 0$ is $\one = \ee^{\nix}$, wherefore we now assume
  $0 < a+b < 1$. The generator
  $A = M - \one = \left( \begin{smallmatrix} -a & a \\ b &
      -b \end{smallmatrix} \right)$ has spectrum
  $\sigma (A) = \{ 0, -(a + b) \}$.  Consequently, $a + b <1$ is its
  spectral radius, and
\[
    Q \, \defeq \, \log (M) \, = \, \log (\one + A) \, = 
    \sum_{m\geqslant 1} \frac{(-1)^{m-1}}{m} \ts A^m
\]   
converges in norm. Since $A^2 = - (a + b) A$, one finds
\begin{equation}\label{eq:log-calc}
   Q \, = \sum_{m\geqslant 1} \frac{(a + b)^{m-1}}{m}
   \ts A \, = \, - \ts \frac{\log (1- a - b)}{a + b} \ts A \ts ,
\end{equation}
which is a Markov generator because the scalar prefactor of $A$ is
always a positive number in our case. Since $M = \ee^Q$ by
construction, any such $M$ is embeddable, and we see that the
necessary determinant condition is also sufficient.
\end{proof}

\begin{remark}
  Our proof of Theorem~\ref{thm:Kendall} was chosen for later
  generalisability.  When $d=2$ and $a+b>0$, a simple alternative
  argument could employ the diagonalisability of $M$ as
 \[
      M \, = \, \myfrac{1}{a+b} \begin{pmatrix}
      1 & -a \\ 1 & b \end{pmatrix} \begin{pmatrix}
      1 & 0 \\ 0 & 1-a-b \end{pmatrix} \begin{pmatrix}
      b & a \\ -1 & 1 \end{pmatrix} ,
 \]  
 which directly gives \eqref{eq:log-calc} by the definition of the
 logarithm of a matrix. This more directly shows that $M$ is
 embeddable if and only if it has positive spectrum.  \exend
\end{remark}

It is instructive, and will be useful later on, to also consider the
complementary point of view to the proof of Theorem~\ref{thm:Kendall},
where one starts from a general rate matrix
$Q = \left( \begin{smallmatrix} -\alpha & \alpha \\ \beta &
    -\beta \end{smallmatrix} \right)$ with $\alpha, \beta \geqslant 0$
and calculates $\ee^{t Q}$.  We may assume $\alpha + \beta >0$ to
avoid the trivial case $\ee^{\nix}=\one$. Now, with
$Q^2 = - (\alpha + \beta) Q$, one explicitly calculates the
exponential series as
\begin{equation}\label{eq:exp1}
  \ee^{t \ts Q} \, = \, \one + \varphi (t) \ts \ts Q
  \quad \text{with} \quad \varphi (t) \, = \,
  \frac{1 - \ee^{-(\alpha+\beta ) t}}{\alpha+\beta} \ts ,
\end{equation}
which also covers the limit
$\alpha+\beta \,\mbox{\footnotesize $\searrow$}\, 0$ via
l'H\^{o}pital's rule.  Since the sum of the two off-diagonal elements
of $\ee^{t \ts Q}$ satisfies
$0 \leqslant 1 - \ee^{- (\alpha+\beta) \ts t} < 1$ for all
$t\geqslant 0$, the spectral radius of $\ee^{t\ts Q} - \one$ is less
than $1$, and one can verify that all Markov matrices $M$ from
Theorem~\ref{thm:Kendall} are covered precisely once. This gives the
following result that is specific to $d=2$; see
\cite{Cuthbert,Cuthbert-2} for some results with $d\geqslant 3$, and
\cite[Cor.~10{\ts}(3) and Thm.~11]{Davies} for a summary.

\begin{coro}\label{coro:unique}
  If a two-dimensional Markov matrix\/ $M$ is embeddable, there is
  precisely one rate matrix\/ $Q$ such that\/ $M=\ee^Q$, namely the
  one from Eq.~\eqref{eq:log-calc}. \qed
\end{coro}

\begin{remark}
  Note that Theorem~\ref{thm:Kendall} is not restricted to irreducible
  Markov matrices. It is worth noting that the irreducible cases are
  also reversible, and that the embeddable ones among them are also
  reversibly embeddable, as discussed in more generality in
  \cite{Chen}. Reversible Markov matrices form another interesting
  class, though we do not consider them here.
  
  The property of reversibility fails to be semi-stable, as
  can be seen from
  $\left(\begin{smallmatrix} 1-\varepsilon & \varepsilon \\ 1 - 2
      \varepsilon & 2 \varepsilon \end{smallmatrix} \right)$ for small
  $\varepsilon > 0$, which is both embeddable and reversible, while
  the limit $\varepsilon \ts\ts
  \raisebox{1pt}{\text{\tiny $\searrow$}} \ts
  \ts 0$ gives the matrix $\left( \begin{smallmatrix} 1 & 0 \\ 1 & 0
  \end{smallmatrix}\right)$, which is neither. However, if
  $(M_n)^{}_{n\in\NN}$ is a converging sequence of reversible,
  irreducible Markov matrices such that the limit is still irreducible,
  reversibility is preserved as well.
  \exend
\end{remark}

\begin{figure}
\includegraphics[width=0.7\textwidth]{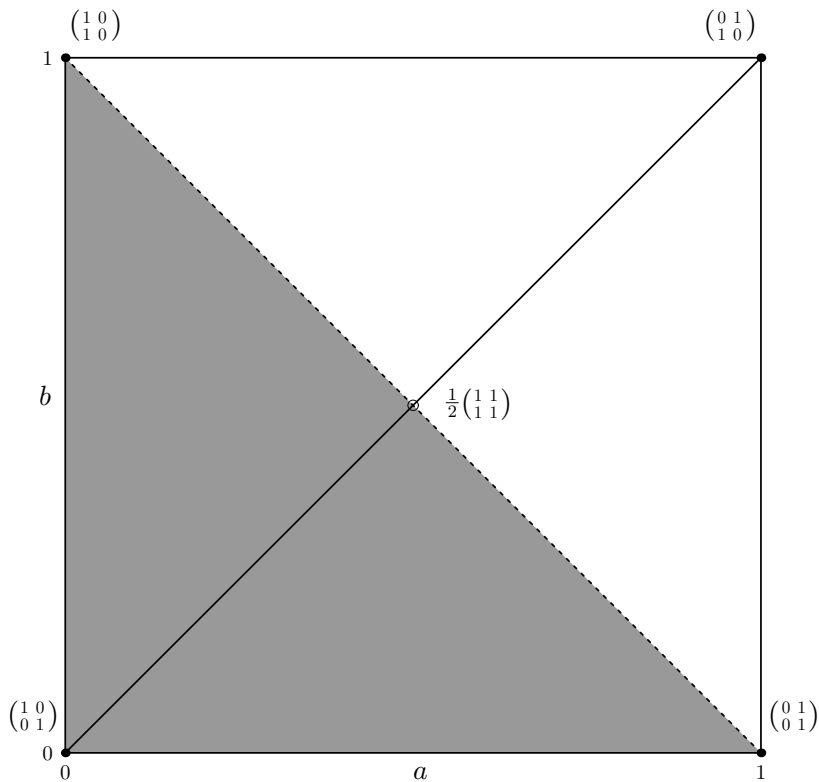}
\caption{The closed convex set $\cM_2$ of Markov matrices for $d=2$,
  described as a square with the parametrisation from
  Theorem~\ref{thm:Kendall}. The four corners are the extremal points,
  which correspond to the stochastic $\{ 0, 1\}\ts$-matrices as
  indicated.  The shaded region represents the subset $E^{}_2$ of
  embeddable matrices, where the dashed line corresponds to the Markov
  matrices with determinant~$0$, which do not belong to
  $E^{}_2 =\cM^{>}_{2}$. The diagonal line is the $1$-simplex of
  symmetric Markov matrices, which agree with the doubly stochastic
  and with the circulant Markov matrices for $d=2$. Note that $E^{}_2$
  is star-shaped with respect to the boundary point
  $\bigl( \frac{1}{2}, \frac{1}{2} \bigr)$, as is the entire set
  $\cM^{}_2$. \label{fig:convex}}
\end{figure}

Although $d=2$ is not really representative for the general case, it
is nevertheless instructive to illustrate the situation from the
viewpoint of convex sets, which we show in Figure~\ref{fig:convex}.
The algebraic situation for $d=2$ is as follows.

\begin{lemma}
  One has\/ $\cE^{}_2 = E^{}_2 =\cM^{>}_{2}$, while\/ $\cE^{+}_{2}$
  consists of all elements of\/ $\cE_2$ with\/ $ab>0$.
\end{lemma}  

\begin{proof}
  If $M$ and $M'$ are embeddable, we have
  $0 < \det (M), \det (M') \leqslant 1$ by
  Theorem~\ref{thm:Kendall}. Then,
  $\det (M M') = \det (M) \det (M') \in ( 0,1]$ as well, which means
  also $M M'$ is embeddable, and $E_2$ is closed under multiplication.
  The second claim follows from
  Proposition~\ref{prop:necessary}{\ts}\eqref{part:four}.
\end{proof}

To get a better understanding in the matrix setting, we take a look at
commutativity. Writing
$M = \left( \begin{smallmatrix} 1-a & a \\ b & 1-b \end{smallmatrix}
\right)$ and $M' = \left( \begin{smallmatrix} 1-a' & a' \\ b' & 1-b'
\end{smallmatrix} \right)$ as before, and setting $A=M-\one$
and $A'=M'-\one$, a simple calculation gives the
following result.

\begin{fact}\label{fact:commute}
  One has\/
  $\, [M, M'] = \nix \; \Longleftrightarrow \; [A,A']=\nix \;
  \Longleftrightarrow \; a \ts\ts b' = a' b$. \qed
\end{fact}

When $M$ and $M'$ are embeddable, so $M=\ee^Q$ and $M'=\ee^{Q'}$, one
has $M M' = \ee^{Q''}$, where $Q'' = Q + Q'$ if $[ Q, Q' ] =\nix $.
Otherwise, the new rate matrix $Q''$ can be calculated by the
Baker--Campbell--Hausdorff (BCH) formula,\footnote{We refer to the
  \textsc{WikipediA} entry on the BCH formula for a quick summary.}
at least in principle. We shall give a simple, closed formula for this
case below in Eq.~\eqref{eq:neq-Q}. One consequence is that $Q''$
belongs to the Lie algebra generated by $Q$ and $Q'$, which motivated
the investigation of Lie--Markov models \cite{Jeremy,SFSJ}. While the
row sums of $Q''$ all vanish, it remains a difficult problem (when
$d \geqslant 3$) to decide when $Q''$ is a Metzler matrix, and hence a
Markov generator.

In general, when $M$ and $M'$ are embeddable but fail to commute, the
embedding semigroups have no simple relation to one another. The
situation, neither restricted to $d=2$ nor to Markov generators, reads
as follows; see \cite{Engel} for background and various extensions.

\begin{lemma}\label{lem:prod}
  For\/ $A, B, C \in \Mat (d,\CC)$, with\/ $d\geqslant 2$,
  the following statements are equivalent.
\begin{enumerate}\itemsep=2pt
\item
  One has\/ $C = A + B$ with\/ $[A,B]=\nix$.
\item  
  One has\/ $\ee^{t A} \ee^{t B} = \ee^{t\ts C}$ for all\/
  $t\in \RR$.
\item  
  One has\/ $\ee^{t A} \ee^{t B} = \ee^{t\ts C}$ for all\/
  $t\geqslant 0$.
\item One has\/ $\ee^{t A} \ee^{t B} = \ee^{t\ts C}$ for some\/
  $\varepsilon > 0$ and all\/ $0 \leqslant t < \varepsilon$.
\item One has\/ $\ee^{t A} \ee^{t B} = \ee^{t\ts C}$ for some\/
  $t_0 \in \RR$, $\varepsilon > 0$ and all\/
  $ \lvert t - t_0 \rvert < \varepsilon$.
\end{enumerate}  
\end{lemma}

\begin{proof}
  Condition (2) follows from (1), as can be checked by a simple
  calculation with the exponential series, and (2) obviously implies
  (3), (4) and (5), where (3) $\Rightarrow$ (4) is also clear.

  Now, assume $\ee^{t A} \ee^{t B} = \ee^{t\ts C}$ for small
  $t\geqslant 0$. Evaluating the time derivative of both sides at
  $t=0$ gives $A+B=C$. With this, now evaluating the second time
  derivative of both sides at $t=0$, one sees that $A$ and $B$ must
  commute, so (4) $\Rightarrow$ (1) $\Rightarrow$ (2).
  
  Finally, assume (5). As $t_0 = 0$ is covered by (4), we may take
  $t_0=1$ without loss of generality, by rescaling all three matrices
  if necessary. This time, evaluating the derivative at $t=1$ gives
  $A \ts \ee^C + \ee^C \nts B = \ee^C C = C \ee^C$ and thus the double
  identity
\begin{equation}\label{eq:trick}
    C \, = \,\ee^{-C} \! A \ee^C + B 
        \, = \, A + \ee^C \nts B \ts \ee^{-C} ,
\end{equation}
  because $\ee^C$ is invertible. The analogous exercise with 
  the second derivative leads to
\[
    C^2 \, = \, A\ts C + C \ee^C \nts B \ee^{-C} 
       \, = \, C\nts A + C \ee^C \nts B \ee^{-C} ,
\]  
where the second equality emerges from Eq.~\eqref{eq:trick} via
multiplying it by $C$ from the left. This implies $[A,C]=\nix$, and
Eq.~\eqref{eq:trick} then gives $C=A+B$.  This, in turn, implies
$[A,B]=\nix$, hence (5) $\Rightarrow$ (1), and we are done.
\end{proof}

\begin{remark}
  Lemma~\ref{lem:prod} is related to rate matrix estimates in
  phylogenetic reconstructions. There, a Markov matrix
  $M = \ee^{t \ts C}$ for time $t$ needs to be consistent with added
  data or measurements at an intermediate time $s$, meaning that
  $\ee^{t \ts C} = \ee^{s A} \ee^{(t-s) B}$ should hold for suitable
  generators $A$, $B$ and $C$. When this is to remain true for some
  small intervals around $t$ and $s$, which may be viewed as some kind
  of stability requirement, arguments as in the previous proof force
  the three generators to be equal. There are variants of this
  situation, the details of which are left to the interested reader.
  Particular aspects, such as the inconsistency of some
  time-reversible models and consequences thereof, are discussed in
  \cite{todo}.  \exend
\end{remark}

Let us next look at an interesting example where a power of a
non-embeddable matrix is embeddable, and how this emerges.

\begin{example}\label{ex:paradox}
  Consider $M = \left( \begin{smallmatrix} \frac{1}{2} & \frac{1}{2} \\
      1 & 0 \end{smallmatrix} \right)$
  from Example~\ref{ex:trivex}, which is primitive because
\[
    M^2 \, = \, \begin{pmatrix}
    \frac{3}{4} & \frac{1}{4} \\
    \frac{1}{2} & \frac{1}{2}
    \end{pmatrix}
\]
is positive, while $M$ itself is not. Consequently, $M$ is not
embeddable, by
\mbox{Proposition~\ref{prop:necessary}{\ts}\eqref{part:four}}, while
$M^2$ certainly is, by an application of Theorem~\ref{thm:Kendall}.

Using the formula from Eq.~\eqref{eq:log-calc}, one finds
$\ee^{2 Q} = M^2$ with
\[
      Q \, = \, \frac{4 \ts \log (2)}{3} \begin{pmatrix}
      -\frac{1}{4} & \frac{1}{4} \\ \frac{1}{2} & - \frac{1}{2}
      \end{pmatrix}   \quad \text{and} \quad
       M' \, = \, \ee^Q \, = \, \begin{pmatrix}
      \frac{5}{6} & \frac{1}{6} \\ \frac{1}{3} &
       \frac{2}{3}\end{pmatrix} ,
\]
where $M'$ is another matrix root of $M^2$, but a  positive
and embeddable one.

The analogous phenomenon happens whenever a Markov matrix fails to
be embeddable, though a power of it is. Recall that the
embeddable Markov matrices are the non-singular Markov matrices
that possess a Markov $n$-th root for every $n\in\NN$; see \cite{King,Joh2}.  \exend
\end{example}

Figure~\ref{fig:convex} also illustrates $\cM_2$ and its subset of
symmetric Markov matrices, some aspects of which will later be
extended in Remark~\ref{rem:I-mat}. With hindsight, the set $\cM_2$
can be related to different matrix algebras, including
$\mathrm{Mat} (2,\RR)$.  While this viewpoint becomes highly complex
for $d>2$, the situation is simpler for another matrix algebra, which
we shall discuss next.

\section{Equal-input matrices}\label{sec:equal-input}

Following \cite[Sec.~7.3.1]{Steel}, let us consider a special, but
practically important, class of Markov matrices, known as
\emph{equal-input} (or Felsenstein \cite{Fels}) matrices.  For its
formulation, let $C$ be a $d\nts\nts \times\! d$-matrix with equal
rows, each being $(c^{}_{1}, \ldots , c^{}_{d})$, and define
$c=c^{}_{1} + \cdots + c^{}_{d}$ as its \emph{summatory parameter},
which will also be called its \emph{parameter sum}. A matrix $C$ with
$c=1$ and all $c_i\geqslant 0$ is Markov, but of rank $1$, and thus
never embeddable for $d>1$.

\begin{fact}
  Any square matrix\/ $C$ with equal rows and summatory parameter\/
  $c$ satisfies the relation\/ $C^2 = c \ts\ts C$.  When\/ $c\ne 0$,
  it is always diagonalisable, while it is nilpotent for\/ $c=0$. In
  the latter case, it is diagonalisable if and only if\/ $C=\nix$.
  \qed
\end{fact}

Now, since $C$ itself is not interesting enough, consider
\begin{equation}\label{eq:equal-input}
     M^{}_{C} \, \defeq \,  (1-c) \ts \one + C \ts ,
\end{equation}
which is a matrix with row sums $1$. It is Markov if $c_i\geqslant 0$
and $c \leqslant 1 + c_i$ for all $i$. We call such Markov matrices
\emph{equal-input}, since they describe a Markov chain where the
probability of a transition $i \to j$, for $i\ne j$, depends only on
$j$. As such, an equal-input Markov matrix $M$ emerges from a matrix
$C$ with equal rows as given in Eq.~\eqref{eq:equal-input}.  We denote
the set of all equal-input Markov matrices with fixed dimension $d$ by
$\cC_d$, so $\cC_d \subseteq \cM_d$.

Since $C$ has spectrum $\sigma (C) = \{ 0, \ldots, 0, c\}$, with
$d \! - \! 1$ copies of $0$, one gets
\begin{equation}\label{eq:def-mc}
    \det (M^{}_{C} ) \, = \, (1-c)^{d-1} \ts ,
\end{equation}
which, for embeddability, has to lie in $(0,1]$ by
Proposition~\ref{prop:necessary}{\ts}\eqref{part:two}.  When
$d \geqslant 2$ is even, Eq.~\eqref{eq:def-mc} then clearly implies
$0\leqslant c < 1$, where $c=0$ with all $c_i \geqslant 0$ means
$c^{}_1 = c^{}_2 = \ldots = c^{}_d =0$ and thus $M^{}_{C} = \one$.
Note that $d=2$ coincides with the case treated in
Theorem~\ref{thm:Kendall}, so equal-input Markov matrices can be
viewed as one generalisation of Section~\ref{sec:two} to higher
dimensions.

For general $d$, whenever $0< c<1$, the spectral radius of
$A = M^{}_{C} -\one$ is less than $1$ and we can use the relation
$A^2 = -c \ts A $ to obtain
\begin{equation}\label{eq:ei-gen}
     \log (M^{}_{C} ) \, = \, - \ts \frac{\log (1-c)}{c} \ts A
\end{equation}
in complete analogy to Eq.~\eqref{eq:log-calc}, where $A = C - c\ts \one$
is a generator of equal-input type, called an \emph{equal-input
  generator} from now on, as it is derived from a matrix
$C$ with equal rows in the obvious way. 
Note that $c=0$ means $A=\nix$, so that the
formula in Eq.~\eqref{eq:ei-gen} also covers this limiting case via
l'H\^{o}pital's rule.

Observe that $M\in\cM_d$ being of equal-input type is a semi-stable,
but not a stable property in the sense of Definition~\ref{def:stable},
so we should expect some subtleties. Thus, consider a general matrix
$B$ with equal rows $(\alpha^{}_1 , \ldots , \alpha^{}_d )$.  Then,
$Q^{}_{\nts B} \defeq B - \alpha \ts \one$ with
$\alpha = \alpha^{}_1 + \cdots + \alpha^{}_d$ is an equal-input
generator whenever all $\alpha_i \geqslant 0$.  With
$Q^{\ts 2}_{\nts B} = -\alpha \, Q^{}_{\nts B}$, one gets
\[
    \ee^{t \ts Q^{}_{\nts B}} \, = \, \one + \varphi (t) \ts Q^{}_{\nts B}
    \quad \text{with} \quad \varphi (t) \, = \,
    \frac{1-\ee^{-\alpha t}}{\alpha} \ts ,
\]
this time in complete analogy to the formula in
Eq.~\eqref{eq:exp1}. In particular, all equal-input Markov matrices
$M^{}_C = (1-c) \ts \one + C \in \cC_d$ with $0\leqslant c <1$ are
embeddable in this way.

Also, observing that 
\begin{equation}\label{eq:ci-rule}
    M^{}_{C} \ts M^{}_{C'} \, = \, M^{}_{C''}
    \quad \text{with} \quad
     C'' \, = \, (1-c{\ts}') \ts C + C'
     \quad \text{and} \quad
     c{\ts}'' \, = \, c + c{\ts}' - c \ts c{\ts}'
\end{equation}
in obvious notation, we know that the product of equal-input Markov
matrices is again equal-input \cite[Lemma~7.2{\ts}(iv)]{Steel}. Also,
$0\leqslant c, c{\ts}' < 1$ implies $c{\ts}'' \in [0,1)$, in line with
the fact that the equal-input Markov matrices form the (closed)
positive cone of a matrix algebra.

Putting these pieces together, we have the following result for even
dimensions.

\begin{prop}\label{prop:equal-input}
  For\/ $d$ even, an equal-input Markov matrix\/ $M^{}_{C}$ is
  embeddable if and only if its parameter sum satisfies\/
  $0 \leqslant c < 1$. The set of such matrices, $\cC_d \cap E_d$,
  forms a monoid. The irreducible elements in\/ $\cC_d \cap E_d$ are
  the  positive ones, and they form a semigroup and a
  two-sided ideal within\/ $\cC_d \cap E_d$.  \qed
\end{prop}

For odd $d$ and $c>0$, the eigenvalues of $M^{}_{C}$ are $1$ and
$1-c$, where the algebraic multiplicity of $1-c$ is $d-1$ and thus
even. Hence, the case $c>1$ can neither be excluded from embeddability
by the determinant criterion nor by
Proposition~\ref{prop:necessary}{\ts}(4); compare
\cite[Prop.~2]{Davies}.  To illustrate this point, let us review and
expand \cite[Ex.~16]{Davies} in our setting.

\begin{example}\label{ex:no-go}
   Consider the commuting Markov generators
\begin{equation}\label{eq:J-def}
     Q \, = \, \begin{pmatrix} -1 & 1 & 0 \\ 0 & -1 & 1 \\
         1 & 0 & -1 \end{pmatrix}
     \quad \text{and} \quad
     J^{}_3 \, = \, \myfrac{1}{3} \begin{pmatrix}
     -2 & 1 & 1 \\ 1 & -2 & 1 \\ 1 & 1 & -2 \end{pmatrix} , 
\end{equation}
with spectra
$\sigma (Q) = \big\{0, - \frac{1}{2} \bigl( 3 \pm \ii
\sqrt{3}\,\bigr)\big\}$
and $\sigma (J^{}_3) = \{ 0, -1 , -1 \}$.  With $J^2_3 = -J^{}_3$, one
gets
\begin{equation}\label{eq:J-exp}
     \ee^{t J_3} \, = \, \one + (1-\ee^{-t}) \ts J^{}_3 \ts .
\end{equation}
The matrix ring $\RR[J_3] = \{ a \ts \one + b J_3 : a,b \in \RR \}$ is
two-dimensional (viewed as a vector space over $\RR$). Now, $J_3$
fails to have simple spectrum, but is diagonalisable, so $\Co (J_3)$
is larger than $\RR [J_3]$ by Fact~\ref{fact:comm}.  If $Q'$ is a rate
matrix, a simple computation shows that $[Q',J_3]=\nix$ forces $Q'$ to
be \emph{doubly stochastic}, which means that also all column sums of
$Q'$ are zero.

Via an explicit calculation, one finds the embeddable matrix
\[
   M \, = \,  \exp \Bigl({\myfrac{2 \pi}
   {\mbox{\small $\sqrt{3}$} } \, Q}\Bigr) \, = \, \begin{pmatrix}
   \frac{1}{3} - 2 \delta & \frac{1}{3} + \delta & \frac{1}{3} + \delta \\
   \frac{1}{3} + \delta & \frac{1}{3} - 2 \delta & \frac{1}{3} + \delta \\ 
   \frac{1}{3} + \delta & \frac{1}{3} + \delta & \frac{1}{3} - 2 \delta 
    \end{pmatrix} \, = \, \one + (1 + 3 \ts \delta) \ts J^{}_3
\]      
with $\delta = \frac{1}{3} \ee^{-\pi \sqrt{3}} \approx 0.00144 > 0$
and $\det (M) = 9 \ts \delta^2 = \ee^{- 2 \pi \sqrt{3}} > 0$.  The
particular property here is that $M$ is a symmetric, equal-input
Markov matrix which has a negative eigenvalue with algebraic
multiplicity $2$,  $\lambda = -\ee^{- \pi \sqrt{3}}$, but summatory
parameter $c = c^{}_{\nts M} = 1 + 3 \delta > 1$. This shows that
Proposition~\ref{prop:equal-input} does not extend to $n$ odd.
Note also that $M=\ee^Q$ with a symmetric $Q$ is impossible
because $\sigma(Q)\subset \RR$ would give a contradiction with
$\lambda < 0$.

Algebraically, we have $Q\in \Co (J_3) \setminus \RR [J_3]$, where $Q$
is doubly stochastic and circulant, see Fact~\ref{fact:J-comm} below
for more, but neither symmetric nor of equal-input type.  Since
$[Q,J_3]=\nix$, a one-parameter family of such examples can be
obtained by using the generator
$\frac{2 \pi}{\sqrt{3}} \ts Q + \varepsilon J_3$ with sufficiently
small $\varepsilon$.

Now, consider a small neighbourhood $U$ of $M$ in $\cM_3$. Since
$M\in E_3$, where $E_3 \subset \cM_3$ is a set of full topological
dimension $6$, the local homeomorphism property of the exponential map
implies that $U$ must contain other embeddable equal-input matrices
with $c>1$, where not all $c_i$ are equal. None of them can be
equal-input embeddable. In this sense, the above example is neither
isolated nor restricted to a lower-dimensional family of matrices.
\exend
\end{example}

The observations on $J_3$ from Example~\ref{ex:no-go} can be
summarised and extended as follows. A generalisation to arbitrary
$d\geqslant 2$ will be stated in Lemma~\ref{lem:ei-comm}.

\begin{fact}\label{fact:J-comm}
  A matrix\/ $B \in \Mat (3,\RR)$ commutes with the Markov generator\/
  $J_3$ from Eq.~\eqref{eq:J-def} if and only if there is a real
  number\/ $\tau$ such that all row and all column sums of\/ $B$
  equal\/ $\tau$. The centraliser of\/ $J_3$ is five-dimensional, and
  any\/ $B\in \Co (J_3)$ is of the form
\[
    B \, = \, \tau \one \, +  \begin{pmatrix}
    -x-y & x+w & y-w \\ x-w & -x-z & z+w \\
    y+w & z-w & -y-z \end{pmatrix}
\]
with\/ $\tau, x,y,z, w \in \RR$.  In particular, a Markov matrix or a
Markov generator commutes with\/ $J_3$ if and only if it is doubly
stochastic.
\end{fact}

\begin{proof}
  The first claim follows from a simple calculation around
  $[B,J_3]=\nix$, which also gives the parametrisation chosen.  The
  last claim is obvious via restricting $\Co (J_3)$ to the set of
  Markov matrices or generators, respectively.
\end{proof}

Clearly, the matrix $M$ from Example~\ref{ex:no-go} cannot be written
as $\ee^Q$ with $Q$ an equal-input generator, as any $\ee^Q$ then has
parameter sum $c < 1$, while we had $c^{}_{\nts M} >1$. Also, due to
Eqs.~\eqref{eq:ei-gen} and \eqref{eq:ci-rule}, equal-input embeddable
Markov matrices are closed under multiplication, and we have the
following counterpart of Proposition~\ref{prop:equal-input}.

\begin{lemma}
  When\/ $d$ is odd, an equal-input Markov matrix\/ $M^{}_{C}$ is
  equal-input embeddable if and only if\/ $0 \leqslant c < 1$, and the
  class of such matrices forms a monoid.  \qed
\end{lemma}

The difference to $d$ even is that there are other embeddable cases,
but not with a generator of equal-input type. This matches with the
fact that being equal-input is only a semi-stable property in the
sense of Definition~\ref{def:stable}.  We thus have the following
general statement.

\begin{theorem}\label{thm:equal-input}
  An equal-input Markov matrix\/ $M^{}_{C}$ is equal-input embeddable
  if and only if its parameter sum satisfies\/ $0 \leqslant c < 1$.
  The class of such matrices forms a monoid, with the subset of
  positive ones being a two-sided ideal in this monoid.
    
  Moreover, when\/ $d$ is even, irrespectively of the type of
  generator, no further equal-input Markov matrices are embeddable,
  while additional embeddable cases do exist for\/ $d$ odd, where\/
  $1-c$ then is a negative eigenvalue with even algebraic
  multiplicity. \qed
\end{theorem}

When $M$ is equal-input embeddable, this does not exclude further
embeddings of a different nature, as we shall see in
Remark~\ref{rem:multi}.  However, more can be said about the
equal-input solutions.  If $Q$ is an equal-input generator, one has
$Q=C - c \ts \one$ and thus
\begin{equation}\label{eq:exp-eig}
  \ee^Q \, = \, \one + \myfrac{1-\ee^{-c}}{c}\ts Q
  \, = \, \myfrac{1-\ee^{-c}}{c} \ts C + \ee^{-c}\ts \one \ts ,
\end{equation}
which means that the summatory parameter $\tilde{c}$ of $\ee^Q$ is
$\tilde{c} = 1 - \ee^{-c}$. Now, let $Q$ and $Q'$ be equal-input
generators, with parameters $c$ and $c^{\ts\prime}$. When
$\ee^Q = \ee^{Q'}$, it is immediate that $c=c^{\ts\prime}$, and
Eq.~\eqref{eq:exp-eig} then implies $Q=Q'$. This gives us the
following uniqueness result, which is a generalisation of
Corollary~\ref{coro:unique}.

\begin{coro}
  If an equal-input Markov matrix\/ $M$ is embeddable as\/ $M=\ee^Q$
  with a generator\/ $Q$ of equal-input type, which is possible if and
  only if\/ $\tilde{c} = c^{}_{\nts M} \in [0,1)$, the generator\/ $Q$
  is unique. For\/ $\tilde{c}\in (0,1)$, it is given by
  Eq.~\eqref{eq:ei-gen}, and by\/ $Q=\nix$ for\/ $\tilde{c}=0$.  \qed
\end{coro}

One can further exploit the monoid property of embeddable equal-input
Markov matrices. If $Q$ and $Q'$ are equal-input generators, one finds
$\ee^{Q} \ee^{Q'} = \ee^{Q''}$ with the new generator
\begin{equation}\label{eq:neq-Q}
\begin{split}  
  Q'' \, & = \, \myfrac{c+c^{\ts\prime}}
  {1-\ee^{-(c+c^{\ts\prime})}} \left( \myfrac{1-\ee^{-c}}{c} \, Q
    + \myfrac{1-\ee^{-c^{\ts\prime}}}{c^{\ts\prime}} \, Q'
    + \myfrac{(1-\ee^{-c})(1-\ee^{-c^{\ts\prime}})}
     {c \ts\ts c^{\ts\prime}} \, Q \ts Q' \right) \\
  & = \, \myfrac{c+c^{\ts\prime}}{c \ts
      \bigl(1-\ee^{-(c+c^{\ts\prime})}\bigr)}
    \bigl( \ee^{-c^{\ts\prime}} (1-\ee^{-c}) \ts Q +
    \myfrac{c}{c^{\ts\prime}}\ts (1-\ee^{-c^{\ts\prime}})
    \ts Q' \ts \bigr) ,
\end{split}    
\end{equation}
where the second line follows from the identity
$Q\ts Q' = - c^{\ts\prime} Q$.  When $[Q, Q']=\nix$, which is
equivalent to $c^{\ts\prime} C = c\ts\ts C'$, Eq.~\eqref{eq:neq-Q}
simplifies to $Q'' = Q + Q'$, in line with Lemma~\ref{lem:prod}. In
any case, the summatory parameters of the exponentials are related by
\[
  \tilde{c}^{\ts\prime\prime} \, = \, \tilde{c} +
  \tilde{c}^{\ts\prime} - \tilde{c} \ts\ts \tilde{c}^{\ts\prime}
  \, = \, 1 - \ee^{-(c + c^{\ts\prime})}
\]
in accordance with Eq.~\eqref{eq:ci-rule}.

\begin{remark}\label{rem:constant-input}
  The class of equal-input Markov matrices contains a subclass,
  defined by the condition $c^{}_{1} = c^{}_{2} = \ldots = c^{}_{d}$,
  called \emph{constant-input} Markov matrices. For $d=4$, they in
  particular cover the \emph{Jukes--Cantor} mutation matrices
  \cite{JC}.  Constant-input Markov matrices are closed under
  multiplication, so the previous analysis can be restricted to this
  subclass. Since $M$ from Example~\ref{ex:no-go} is a constant-input
  Markov matrix, we get the result of Theorem~\ref{thm:equal-input}
  also for this subclass, with the same distinction between $d$ even
  and $d$ odd.  We shall meet this class again in
  Section~\ref{sec:c-input}.  \exend
\end{remark}

Let $C$ be as above, with summatory parameter $c$ together with
$c_i \geqslant 0$ for all $1\leqslant i \leqslant d$, and consider the
equal-input generator $Q^{}_{C} = C - c \ts \one$. Since
$Q^{}_{C} = \nix$ when $c=0$, let us assume $c>0$. Then, we have
$C\ne \nix$ and $Q^{}_{C} \, C = \nix$, so $Q^{}_{C}$ has minimal
polynomial $z \ts \ts (z+c)$. In particular, $Q^{}_{C}$ is
diagonalisable, with eigenvalues $0$ and $-c$, the latter with
multiplicity $d \nts - \nts 1$.  In this context, one has the
following extension of Fact~\ref{fact:one-dim}.

\begin{coro}\label{coro:eq-input-gen}
  Let\/ $Q$ be a Markov generator that satisfies any of the equivalent
  properties of Fact~\textnormal{\ref{fact:one-dim}}, say\/
  $Q^2 = - c \ts Q$ with\/ $c>0$. Then, $0$ is a simple eigenvalue
  of\/ $Q$ if and only if\/ $Q$ is an equal-input generator with
  summatory parameter\/ $c>0$.
\end{coro}

\begin{proof}
  Under the assumptions, we have $Q (Q+c\ts \one)=\nix$.  When
  $0\in\sigma (Q)$ is simple, with eigenvector
  $u = (1,1, \ldots , 1)^T$, we see that each column of $Q+c\ts\one$
  must be a scalar multiple of $u$. If the $i$-th column is
  $c^{}_{i} \ts u$ with $c^{}_{i} \in \RR$, we have
  $Q = C - c\ts\one$, which is a generator if and only if all
  $c^{}_{i} \geqslant 0$ together with
  $c = c^{}_{1} + \cdots + c^{}_{d}$. Since $0$ is simple, we must
  have $c>0$.
  
  The converse direction was shown above.
\end{proof}

For $d\geqslant 3$, an equal-input generator inevitably has multiple
eigenvalues such that its minimal and characteristic polynomials
differ. Consequently, by Fact~\ref{fact:comm}, the centraliser is
non-Abelian. Interestingly, on the level of generators, one has the
following elementary result, which can be seen as a generalisation of
Fact~\ref{fact:J-comm}.

\begin{lemma}\label{lem:ei-comm}
  Let\/ $Q = C \! - \! c \ts \one$ be an equal-input generator, with
  non-negative parameters\/ $c^{}_{1}, \ldots , c^{}_{d}$ and\/
  $c = c^{}_{1} + \cdots + c^{}_{d}$. Then, a Markov generator\/ $X$
  commutes with\/ $Q$ if and only if\/
  $(c^{}_{1} , \ldots , c^{}_{d}) X = 0$.  When\/ $c>0$, this is
  equivalent to saying that\/
  $\frac{1}{c} (c^{}_{1} , \ldots , c^{}_{d})$ is an equilibrium state
  of\/ $X$.
\end{lemma}

\begin{proof}
  When $c=0$, which means $Q=\nix$ and
  $c^{}_{1} = \cdots = c^{}_{d} = 0$, every generator $X$ commutes
  with $Q$ and the statement is trivial. Next, assume $c>0$ and
  $Q = C - c\ts \one$. A matrix commutes with $Q$ if and only if it
  commutes with $C$. When $X$ is a generator, one has $XC=\nix$
  because $C$ has constant columns. Then, $[X,C] = \nix$ means
  $CX=\nix$, each row of which is the condition stated. Since
  $\frac{1}{c} (c^{}_{1} , \ldots , c^{}_{d})$ is a probability vector
  for $c>0$, the last claim is clear.
\end{proof}

Notice that such a generator $X$ may have several equilibrium
states, which form a convex set.  Then, when $c>0$, the
characterisation from Lemma~\ref{lem:ei-comm} specifies only one of
them.

\begin{remark}\label{rem:I-mat}
  There is a doubly stochastic, constant-input Markov matrix,
  $I_d \defeq \frac{1}{d}\ts C$ with parameters
  $c^{}_{1} = \ldots = c^{}_{d} =1$, which has several interesting
  properties. Clearly, $\det (I_d) =0$, so $I_d$ is not
  embeddable, but lies on the boundary of $E_d$. This is also clear
  from the relation $I_d = \lim_{t\to\infty} \ee^{t Q}$ with $Q$ being
  any irreducible, doubly stochastic generator in $d$ dimensions.
  More importantly, as can be shown by the methods used for
  \cite[Thm.~2.7]{Joh3}, the subsets of $E_d$ of equal-input or of
  certain doubly stochastic matrices are \emph{star-shaped} with respect to
  $I_d$, which is visible in Figure~\ref{fig:convex} for $d=2$. Since
  $I_d$ is also circulant and symmetric, this will be useful for
  later deformation arguments.  \exend
\end{remark}

\section{Circulant matrices}\label{sec:circulant}

A $d \! \times \! d$-matrix is called \emph{circulant} if each of its
rows emerges from the previous one by cyclically shifting it one
position to the right (see below for more).  Such matrices have a rich
theory of their own; see \cite{PJD} for a detailed exposition. Here,
we are interested in circulant matrices that are also Markov matrices
or generators, respectively. The interested reader may consult
\cite[Ch.~7.3.2]{Steel} for how these matrices fit into the setting of
`group-based models', and \cite{SW} for an extension to
`semigroup-based models'. The latter classification in particular 
includes the equal-input models discussed in 
Section~\ref{sec:equal-input}.

For $d=2$, a Markov matrix is circulant precisely when it is a Markov
matrix of constant-input type, which means
$M=\left( \begin{smallmatrix} 1-a & a \\ a & 1-a \end{smallmatrix}
\right)$. By Theorem~\ref{thm:Kendall}, such an $M$
is embeddable if and only
if $0 \leqslant a < \frac{1}{2}$. In fact, with
$Q = \alpha \left( \begin{smallmatrix} -1 & 1 \\ 1 & -1
\end{smallmatrix} \right)$, one can rewrite our earlier
formula as
\begin{equation}\label{eq:2-circ}
    \ee^{t Q} \, = \, \ee^{-\alpha t}  \cosh (\alpha t) \ts \one
    + \ee^{-\alpha t} \sinh (\alpha t) \begin{pmatrix}
    0 & 1 \\ 1 & 0 \end{pmatrix}
    \, = \, \one + \myfrac{1-\ee^{-2 \alpha t}}{2}
    \begin{pmatrix} -1 & 1 \\ 1 & -1 \end{pmatrix} .
\end{equation}
In particular, whenever a circulant $M$ with $d=2$ is embeddable,
Corollary~\ref{coro:unique} together with Eq.~\eqref{eq:log-calc}
implies that this is only possible with a circulant generator, which
also fits Proposition~\ref{prop:extra}. The explicit computations rely
on the formula $\ee^x = \cosh (x) + \sinh (x)$, which can be seen as a
decomposition of the exponential function over the cyclic group $C_2$.

Let us look at circulant matrices in more generality.  By
\cite[Thm.~3.1.1]{PJD}, $B\in\Mat (d,\CC)$ is circulant if and only if
$B$ commutes with the permutation matrix
\begin{equation}\label{eq:def-P}
    P \, = \, \left( \begin{array}{c | c @{\;\;} c @{\,\;} c}
    0  &  &  & \\
    \vdots &  & \one^{}_{d-1} & \\
    0 & & & \\ \hline
    1 & 0 & \cdots & 0
    \end{array} \right),
\end{equation}
which is the standard $d$-dimensional representation of the cyclic
permutation $(1 2 \ldots d)$. Since $P$ has simple spectrum, namely
$\sigma (P) = \{ \ee^{2 \pi \ii m/d} : 0 \leqslant m \leqslant d \! -
\! 1 \}$, and since $P$ is also a doubly stochastic matrix, as is any
power of it, one finds the following consequence.

\begin{fact}
  The circulant matrices in\/ $\Mat (d,\CC)$, respectively in\/
  $\Mat (d,\RR)$, are the elements of the Abelian matrix ring\/ $\CC[P]$,
  respectively\/ $\RR[P]$. The\/ $d$-dimensional, circulant Markov
  matrices are precisely the convex combinations of\/
  $\one, P, P^2, \ldots , P^{d-1}$, which form a simplex of
  dimension\/ $d\! - \! 1$ within\/ $\Mat (d,\RR)$, and also a monoid
  under matrix multiplication. In particular, every circulant Markov
  matrix is doubly stochastic. \qed
\end{fact}

Clearly, the exponential of a circulant matrix is again circulant.
Moreover, being circulant is a stable property. This can be seen by a
deformation argument on the basis of Remark~\ref{rem:I-mat} as
follows.  Given $M$, select a circulant matrix $M'$ with simple
spectrum near the matrix
$I_d = \frac{1}{d} (\one + P + \ldots + P^{d-1})$ from
Remark~\ref{rem:I-mat} such that $ \alpha M' + (1-\alpha) M$ is
embeddable for all $\alpha\in [0,1]$, which is possible, and apply
\cite[Cor.~6]{Davies} to get the approximation property.  By
Proposition~\ref{prop:extra}, we thus get the following consequence.

\begin{coro}\label{coro:circ}
  Any embeddable circulant Markov matrix is circulant-embeddable.
  When its centraliser is Abelian, it is only circulant-embeddable.
  \qed
\end{coro}

For $d\in\NN$, let $C_d = \ZZ/\nts d \ts \ZZ$ be the cyclic group of
order $d$, represented as $C_d = \{ 0,1, \ldots, d \nts - \! 1\nts\}$ with
addition modulo $d$, and let $\chi^{}_{m} (j) \defeq \omega^{m j}$
with $m,j\in C_d$ and $\omega = \ee^{2 \pi \ii/d}$ define the
\emph{characters} $\chi^{}_{m}$ of $C_d$, which satisfy the well-known
orthogonality relations \cite[Thm.~16.4]{JL}
\[
    \myfrac{1}{d} \sum_{m=0}^{d-1} 
       \overline{\chi^{}_{m}} (k)\ts \chi^{}_{m} (\ell)
       \, = \, \delta^{}_{k,\ell} 
       \quad \text{ and } \quad
       \myfrac{1}{d} \sum_{k=0}^{d-1} 
       \overline{\chi^{}_{m}} (k)\ts \chi^{}_{n} (k)
       \, = \, \delta^{}_{m,n} \ts .
\]
With this, some elementary calculations establish the 
following result.

\begin{fact}\label{fact:cyclic-deco}
  Let\/ $\omega = \ee^{2 \pi \ii /d}$.  Then, the functions\/
  ${\displaystyle f^{(d)}_m} \! : \, \RR \xrightarrow{\quad} \CC$
  with\/ $m\in C_d$, defined by
\[
     t \, \mapsto \,  {\displaystyle f^{(d)}_m (t)}
     \, \defeq \, \myfrac{1}{d}
     \sum_{\ell=0}^{d-1} \overline{\chi^{}_{m}}
     (\ell) \ts \ee^{\omega^{\ell} t} ,
\]   
  satisfy the following properties.
\begin{enumerate}\itemsep=2pt
\item For any\/ $r\in C_d$, one has the decomposition\/
  $\ee^{\omega^{r} t} = \sum_{m=0}^{d-1} {\displaystyle \chi^{}_{m}
    (r) f^{(d)}_{m} (t)}$.
  In particular, this gives\/
  $\ee^{t} = f^{(d)}_0 (t) + f^{(d)}_1 (t) + \cdots + f^{(d)}_{d-1}
  (t)$ for\/ $r=0$.
\item For any\/ $m\in C_d$, the function\/ $f^{(d)}_{m}$ possesses a
  globally convergent Taylor series, namely\/
  ${\displaystyle f^{(d)}_{m} (t) } = \sum_{\ell\geqslant 0}
  \frac{t^{\ell d + m}}{(\ell d + m)!} $, and is thus real-valued.
  \qed
\end{enumerate}  
\end{fact}

Due to the first property, this can be considered as a decomposition
of the exponential function over the cyclic group $C_d$, thus
generalising the earlier case, $d=2$. If $d \ts |D$, so $D = k d$ with
$k\in\NN$, one also has the summatory identity
\begin{equation}\label{eq:fun-sum}
   \sum_{\ell=0}^{k-1} f^{(D)}_{\ell d + r}
   \, = \, f^{(d)}_{r}
\end{equation}
for $r\in C_d$. We shall return to these functions after an explicit
treatment of $d=3$ and $d=4$.

\subsection{A cyclic model for $d=3$}

Let $P$ be the matrix from Eq.~\eqref{eq:def-P} for $d=3$, 
\[
     P \, = \, \begin{pmatrix} 0 & 1 & 0 \\ 0 & 0 & 1 \\
     1 & 0 & 0 \end{pmatrix} ,
\]
which has spectrum $\sigma (P) = \{ 1, \omega , \omega^2\}$ with
$\omega = \ee^{2 \pi \ii/ 3} = - \frac{1}{2} + \ii \frac{\sqrt{3}}{2}$
a primitive third root of unity, so $\omega^2 = \overline{\omega}$ and
$P^3 = \one$.  Also, by Fact~\ref{fact:comm}, one has
$\Co (P) = \RR [P] = \langle \one, P, P^2 \rangle^{}_{\RR}$. Now, $P$
is diagonalisable as
$U^{-1} P \ts\ts U = \diag (1,\omega, \overline{\omega} \ts )$ with
the unitary Fourier matrix\footnote{Note that our $U$ actually is
the complex conjugate of the matrix used for discrete
Fourier transform.} 
\[
     U \, = \, \myfrac{1}{\sqrt{3}} \begin{pmatrix} 1 & 1 & 1 \\
     1 & \omega & \overline{\omega} \\ 1 & 
     \overline{\omega} & \omega
     \end{pmatrix} .
\]
Next, define two Markov generators as $K_1 = P - \one$, which equals
$Q$ from Eq.~\eqref{eq:J-def}, and $K_2 = P^2 - \one$.  They satisfy
\[
    K^{2}_{1} \, = \, K^{}_{2} - 2 \ts K^{}_{1} \, , \quad
    K^{2}_{2} \, = \, K^{}_{1} - 2 \ts K^{}_{2} \, , \quad
    K^{}_{1} K^{}_{2} \, = \, K^{}_{2} K^{}_{1}
    \, = \, - (K^{}_{1} + K^{}_{2} )
\]
and thus generate a two-dimensional matrix algebra over $\RR$, which
is an Abelian subalgebra of $\cA_0$ from Eq.~\eqref{eq:A0}. This
subalgebra consists of the matrices
\[
    \alpha K_1 + \beta K_2 \, = \, \begin{pmatrix}
    -\alpha-\beta & \alpha & \beta \\ \beta & -\alpha-\beta & \alpha \\
    \alpha & \beta & - \alpha - \beta \end{pmatrix} 
\]
with $\alpha, \beta \in \RR$, which are all circulant.  Since the rate
matrices $K_1$ and $K_2$ commute, one has
$\ee^{\alpha K_1 + \beta K_2} = \ee^{\alpha K_1} \ee^{\beta K_2}$,
wherefore it is natural to set
\[
    D_1 \, \defeq \, U^{-1} K_1 \ts U \, = \, \diag (0,
     \omega \nts - \nts 1, \overline{\omega} \nts - \nts 1)
    \quad \text{and} \quad
    D_2 \, \defeq \, U^{-1} K_2 \ts\ts U \, = \, \diag (0, 
     \overline{\omega} \nts - \nts 1, \omega \nts - \nts 1) \ts .
\] 
Now, with $\alpha, \beta \in \RR$ and
$\eta \defeq \ee^{\alpha\ts \omega + \beta \ts \overline{\omega} -
  (\alpha + \beta)}$, we get
\begin{equation}\label{eq:circulant}
\begin{split}
   \exp (\alpha K_1 + \beta K_2) \, & = \, U 
   \exp (\alpha D_1 + \beta D_2) \, U^{-1} 
   \, = \, U \diag ( 1 , \eta, \overline{\eta} \ts\ts ) 
        \, U^{-1}  \\[2mm]
   & = \, \one + x (\alpha, \beta) K_1 + y (\alpha, \beta ) K_2
\end{split}   
\end{equation}
with
$x (\alpha, \beta) = \frac{1}{3} \bigl( 1 + 2 \ts \ee^{-\gamma} \cos
\bigl( \delta - \frac{2 \pi}{3}\bigr)\bigr)$ and
$y (\alpha, \beta) = x( \beta, \alpha)$, where
$\gamma = \frac{3}{2} (\alpha + \beta)$ and
$\delta = \frac{\sqrt{3}}{2} (\alpha - \beta)$.  Note that
$ x (\alpha, \beta) + y (\alpha , \beta ) = \frac{2}{3} \bigl( 1 -
\ee^{-\gamma} \cos (\delta)\bigr)$.  Another way to write the result,
with $\psi = \frac{2 \pi}{3}$, is
\[
    \exp (\alpha K_1 + \beta K_2) \, = \,
    \myfrac{1}{3} \begin{pmatrix} 1 & 1 & 1 \\
    1 & 1 & 1 \\ 1 & 1 & 1 \end{pmatrix} + 
    \myfrac{2 \ts \ee^{-\gamma}}{3} \begin{pmatrix}
    \cos (\delta) & \cos (\delta - \psi ) &
    \cos (\delta + \psi) \\ 
     \cos (\delta + \psi) & \cos (\delta ) &
    \cos (\delta - \psi) \\  
    \cos (\delta - \psi) & \cos (\delta + \psi ) &
    \cos (\delta ) \end{pmatrix} ,
\]
where $\cos (\delta) + \cos (\delta+\psi) + \cos (\delta-\psi) =0$ 
holds for all $\delta$. 

\begin{figure}
\includegraphics[width=0.6\textwidth]{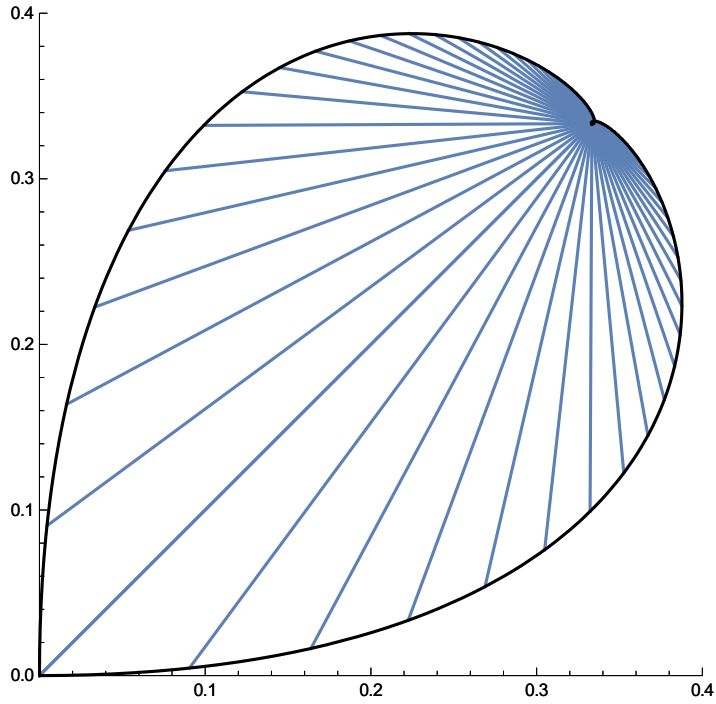}
\caption{Sketch of the parameter region for embeddable circulant
  Markov matrices with $d=3$. Note that the point
  $\bigl(\frac{1}{3},\frac{1}{3}\bigr)$ does not belong to this
  region, which is star-shaped relative to this point; see text for
  further explanations.\label{fig:circulant}}
\end{figure}

To see which Markov matrices of the form
$M(a,b) \, \defeq \, \one + a K_1 + b K_2 $ are realised by
Eq.~\eqref{eq:circulant}, we need to determine the image of
$\RR^{2}_{\geqslant 0}$ under the mapping
\[
   (\alpha, \beta)  \mapsto z(\alpha,\beta) 
   \defeq \bigl(x(\alpha,\beta), y(\alpha,\beta)\bigr).
\]
The necessary determinant condition 
$\det (M(a,b)) = 1 - 3 (a+b) + 3 (a^2 + a b + b^2) \in (0,1]$
from Proposition~\ref{prop:necessary}{\ts}\eqref{part:two}
implies the inequality
\[
    a + b - \myfrac{1}{3} \, < \,
    a^2 + a b + b^2 \, \leqslant \, a+b  \ts ,
\]
which is satisfied, but not sufficient in this case.
Figure~\ref{fig:circulant} illustrates the region that is covered,
where the special point belongs to $I^{}_{3} = M \bigl(
\frac{1}{3},\frac{1}{3} \bigr)$. The region
is bounded by the two curves
\begin{equation}\label{eq:boundary}
     \{ z ( \alpha ,0) : \alpha \geqslant 0 \}
     \quad \text{ and } \quad
     \{ z (0, \beta ) : \beta \geqslant 0 \} \ts ,
\end{equation}
which both start at $(0,0)$ and approach the limit point
$\bigl( \frac{1}{3}, \frac{1}{3} \bigr)$ without reaching it. The
entire region is filled with straight lines from the boundary points
towards this special point, as parametrised by
$\{ z (\alpha + t , t ) : t \geqslant 0 \}$ with $\alpha \geqslant 0$
(lower half) and $\{ z (t, \beta+t ) : t \geqslant 0 \}$ with
$\beta \geqslant 0$ (upper half). The special role of the limit point
also emerges from
\[
    \det \begin{pmatrix} x^{}_{\alpha} (\alpha, \beta) &
    y^{}_{\alpha} (\alpha, \beta) \\ x^{}_{\nts\beta} (\alpha, \beta) &
    y^{}_{\nts\beta} (\alpha, \beta) \end{pmatrix} \, = \,
    \ee^{-3 (\alpha + \beta)} \, 
    \xrightarrow{\, \alpha, \beta \to \infty \,} \, 0 ,
\]
which shows that the limit point is the (somewhat degenerate) envelope
\cite{Fer} of our family of curves, where the determinant criterion is
due to Leibniz and Taylor. We can summarise our derivation as follows.

\begin{theorem}
  For a general circulant Markov matrix\/ 
  $M (a,b) = \one + a K_1 + b K_2 \in \cM_3$,
  the following statements are equivalent.
  \begin{enumerate}\itemsep=2pt
  \item $M (a,b)$ is embeddable.
  \item $M (a,b)$ is circulant-embeddable.
  \item The parameter pair\/ $(a,b)$ lies in the closed region
    defined by the bounding curves from Eq.~\eqref{eq:boundary},
    excluding the point\/ $\bigl( \frac{1}{3}, \frac{1}{3}\bigr)$,
    which parametrises the non-embeddable matrix\/ $I_3$ from
    Remark~\textnormal{\ref{rem:I-mat}}.
  \end{enumerate}
\end{theorem}

\begin{proof}
  In view of our above calculations,
  $(3) \Longleftrightarrow (2) \Rightarrow (1)$ is clear, while
  $(1) \Rightarrow (2)$ follows from Corollary~\ref{coro:circ}.
\end{proof}

\begin{remark}
  Circulant Markov matrices with $d=3$ have previously been considered
  in some detail in \cite[Sec.~III]{Len}. This class is particularly
  interesting from the point of view of \emph{multiple} embeddings of
  a given Markov matrix. In smaller and smaller regions of the
  parameter space near $\bigl( \frac{1}{3}, \frac{1}{3}\bigr)$, 
  one can find matrices with an increasing number of
  distinct embeddings. We shall return to this point in
  Section~\ref{sec:c-input}. \exend
\end{remark}

\subsection{The cases $d\geqslant 4$}

For $d=4$, the most general circulant generator reads
\[
    Q \, = \, \begin{pmatrix} 
    -\alpha-\beta-\gamma & \alpha & \beta & \gamma \\
    \gamma & -\alpha-\beta-\gamma & \alpha & \beta  \\ 
    \beta & \gamma & -\alpha-\beta-\gamma & \alpha \\ 
     \alpha & \beta & \gamma & -\alpha-\beta-\gamma 
     \end{pmatrix}
     \quad \text{with} \quad
     \alpha, \beta, \gamma \geqslant 0 \ts .
\]
Its exponential is $\ee^Q = \one + x K_1 + y K_2 + z K_3 \eqdef M(x,y,z)$,
where $K_m \defeq P^m - \one$ with $P$ the permutation matrix of
$(1234)$ in analogy to the previous section, together with\footnote{The 
  validity of Eq.~\eqref{eq:4-form} can easily be verified by means of a
  standard computer algebra program, the details of which are left to
  the interested reader.}
\begin{equation}\label{eq:4-form}
\begin{split}
  x \, & = \, \myfrac{1}{2} \ee^{-(\alpha+\gamma)}
          \bigl( \sinh (\alpha + \gamma) + \ee^{-2 \beta} 
            \sin (\alpha - \gamma)\bigr) , \\
  y \, & = \, \myfrac{1}{2} \ee^{-(\alpha+\gamma)}
          \bigl( \cosh (\alpha + \gamma) - \ee^{-2 \beta} 
            \cos (\alpha - \gamma)\bigr) , \\
  z \, & = \, \myfrac{1}{2} \ee^{-(\alpha+\gamma)}
          \bigl( \sinh (\alpha + \gamma) - \ee^{-2 \beta} 
            \sin (\alpha - \gamma)\bigr) .
\end{split}            
\end{equation}
In the limit $\beta \to \infty$, one has $z=x$ and
$x+y=\frac{1}{2}$. Taking $\beta\to\infty$, the parametrisation
reduces to
\[
   \begin{pmatrix} x \\ y \\ z \end{pmatrix}
   \, = \, \myfrac{1}{4} \begin{pmatrix} 1 \\ 1 \\ 1
   \end{pmatrix} - \myfrac{\ee^{-2 (\alpha + \gamma)}}{4}
   \begin{pmatrix} 1 \\ -1 \\ 1 \end{pmatrix}
\]
and thus covers the line from $\bigl( 0,\frac{1}{2}, 0 \bigr)$ to
$\bigl( \frac{1}{4},\frac{1}{4},\frac{1}{4}\bigr)$, with
$M \bigl( \frac{1}{4},\frac{1}{4},\frac{1}{4}\bigr) = I^{}_{4}$.
Note that no
point on this line is ever reached with finite parameter values.
Otherwise, letting $\alpha \to \infty$ or $\gamma \to \infty$ each
means
$(x,y,z) \xrightarrow{\quad} \bigl( \frac{1}{4}, \frac{1}{4},
\frac{1}{4} \bigr)$, again without ever reaching this point.  In fact,
since
\[
    \det \begin{pmatrix} 
      x^{}_{\alpha} & y^{}_{\alpha} & z^{}_{\alpha} \\
      x^{}_{\beta} & y^{}_{\beta} & z^{}_{\beta} \\
      x^{}_{\gamma} & y^{}_{\gamma} & z^{}_{\gamma} 
    \end{pmatrix} \, = \, \ee^{-4 (\alpha + \beta + \gamma)}
    \, \xrightarrow{\, \beta \to \infty \,} \, 0 \ts ,
\]
we see that the above line is the (again degenerate)
envelope of our family of curves in $\RR^3$.

In general, one always has
$x+z = \frac{1}{2} \bigl( 1 - \ee^{-2 (\alpha + \gamma)} \bigr)$.
Setting $\alpha=\gamma=0$, one obtains the line
$\big\{\bigl(0,\frac{1}{2} ( 1 - \ee^{-2 \beta}),0 \bigr) : \beta
\geqslant 0 \big\}$, without reaching $(0,\frac{1}{2}, 0)$. One can
now check that the possible values of $(x,y,z)$ fill a region that is
bounded by three finite surface sheets of the form
$\{ (x,y,z): (\alpha, \beta, \gamma) \in \cP_i \}$ with
\begin{equation}\label{eq:circ-4}
  \cP^{}_{1} \, = \, \{ 0 \} \! \times \! \RR^{}_{\geqslant 0} 
      \! \times \! \RR^{}_{\geqslant 0} \, , \quad
  \cP^{}_{2} \, = \, \RR^{}_{\geqslant 0} \! \times \! \{ 0 \}
      \! \times \! \RR^{}_{\geqslant 0} \, , 
  \quad \text{and} \quad
  \cP^{}_{3} \, = \, \RR^{}_{\geqslant 0}  \! \times \!
      \RR^{}_{\geqslant 0} \! \times \! \{ 0 \} \, , 
\end{equation}
where the points on the `seam line' from
$\bigl( 0,\frac{1}{2}, 0 \bigr)$ to
$\bigl( \frac{1}{4},\frac{1}{4},\frac{1}{4}\bigr)$ are never reached.

\begin{figure}
\includegraphics[width=0.6\textwidth]{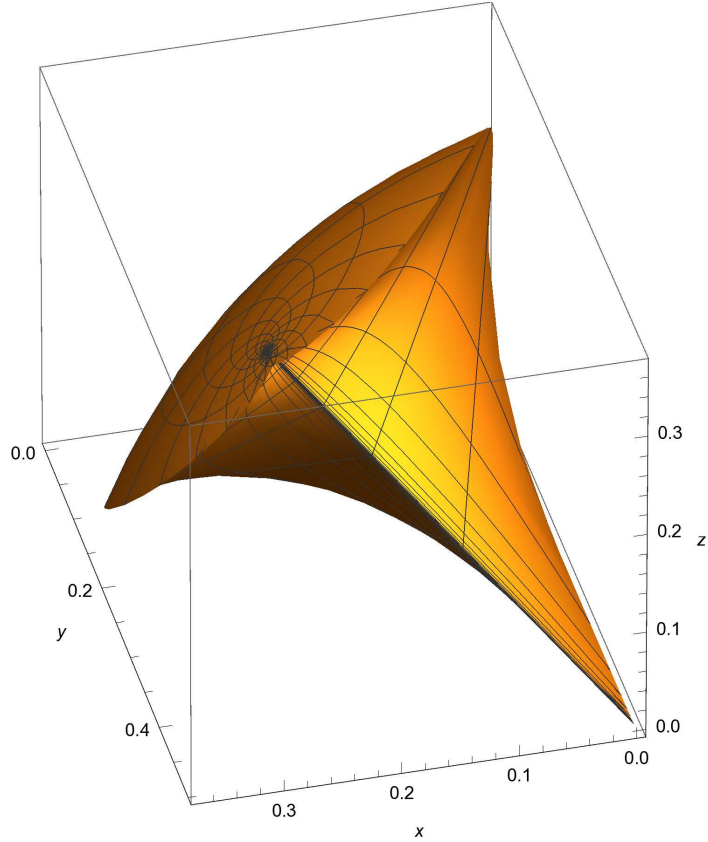}
\caption{Sketch of the parameter region for embeddable circulant
  Markov matrices with $d=4$. The sharp diagonal edge at the front is
  the envelope (or seam line) that does not belong to the region; see
  text for details. \label{fig:circ-4d}}
\end{figure}

We can sum up this analysis as follows, with an accompanying
illustration in Figure~\ref{fig:circ-4d}.

\begin{theorem}
  For a general circulant Markov matrix\/
  $M (x,y,z) = \one + x K_1 + y K_2 + z K_3$ in\/ $\cM_4$, 
  the following statements are equivalent.
\begin{enumerate}\itemsep=2pt
\item $M$ is embeddable.
\item $M$ is circulant-embeddable.
\item The parameters\/ $x,y,z$ lie in the closed region defined by the
  three surface sheets with the parameter regions from
  Eq.~\eqref{eq:circ-4}, without the points on the envelope.  \qed
\end{enumerate}
\end{theorem}

In general, consider $C_d$ and let $n_r $ be the order of the cyclic
subgroup that is multiplicatively generated by $r$, where $r\in C_d$
and $n_r | d$. Now, it is not hard to compute
\begin{equation}\label{eq:f-help}
    \ee^{\alpha P^r} = \, \sum_{\ell=0}^{d-1}
    f^{(d)}_{\ell} (\alpha) \, P^{r \ell} 
    \, = \sum_{m=0}^{n_r - 1}
    f^{(n_r)}_{m} (\alpha) \, P^{mr} ,
\end{equation}
where the second representation is a consequence of
Eq.~\eqref{eq:fun-sum}.  Define $K_j = P^j - \one$ for $j \in C_d$,
where $K_0=\nix$.  All $K_j$ are Markov generators, with relations
$K_i K_j = K_{i+j} - K_i - K_j$. Consequently, the algebra generated
by the $K_j$ is Abelian and of dimension $d \nts\nts - \! 1$.

Now, with
$\alpha^{}_{0} \defeq - \bigl(\alpha^{}_{1} + \cdots + \alpha^{}_{d-1}
\bigr)$, one obtains
\[
    \exp \biggl(\, \sum_{i=1}^{d-1} \alpha^{}_{i} K^{}_{i} \biggr)
    \, = \, \prod_{i=1}^{d-1} \ee^{\alpha^{}_{i} K^{}_{i}} \, = \,
    \ee^{\alpha^{}_{0}} \prod_{i=1}^{d-1} \ee^{\alpha^{}_{i} P^{\ts i}}
    \, = \, \ee^{\alpha^{}_{0}} \sum_{r=0}^{d-1} a^{}_r \ts P^{\ts r}
\]
with the coefficients
\begin{equation}\label{eq:A-def}
  a^{}_r  \, = \, a^{}_{r} \bigl( \alpha^{}_{1}, \ldots, \alpha^{}_{d-1}\bigr)
  \, = \sum_{\substack{m^{}_{1}, \ldots , m^{}_{d-1} = \ts 0 \\
      \sum_{i=1}^{d-1} i \ts m^{}_i \ts\ts \equiv \ts\ts r \, (d)}}^{d-1}\,
  \prod_{j=1}^{d-1} f^{(d)}_{m^{}_j} (\alpha^{}_j) \ts .
\end{equation}
With Fact~\ref{fact:cyclic-deco}{\ts}(1), one finds the relation
\[
     \sum_{r=0}^{d-1} a^{}_{r} \, = \, \prod_{i=1}^{d-1}
     \sum_{m^{}_i=0}^{d-1} f^{(d)}_{m^{}_i} (\alpha^{}_{i}) 
     \: = \, \prod_{i=1}^{d-1} \ee^{\alpha^{}_{i}} \, = \,
     \ee^{-\alpha^{}_{0}} ,
\]
which leads to the alternative expression
\[
    \exp \biggl(\, \sum_{i=1}^{d-1} \alpha^{}_{i} K^{}_{i} \biggr)
    \, = \, \one \ts\ts + \ts\ts \ee^{\alpha^{}_{0}}  \!
    \sum_{r=1}^{d-1} a^{}_{r} \ts K^{}_{r} \ts .
\]
This can be seen as the natural generalisation of our previous
calculations, in particular Eq.~\eqref{eq:2-circ} and the formulas for
$d=3$ and $d=4$. From here, one obtains general criteria for circulant
Markov matrices and their embeddability as follows.

\begin{theorem}
  For fixed\/ $d\geqslant 2$, the most general circulant Markov matrix
  is of the form\/
  $M (x^{}_{1}, \ldots , x^{}_{d-1}) = \one \ts\ts + \ts
  \sum_{r=1}^{d-1} x^{}_{r} \ts K^{}_{r}$, with all\/
  $x^{}_{r} \geqslant 0$ and\/
  $x^{}_{1} + \cdots + x^{}_{d-1} \leqslant 1$. It is embeddable if
  and only if there are non-negative numbers\/
  $\alpha^{}_{1}, \ldots, \alpha^{}_{d-1}$ such that\/
  $x^{}_{r} \ts\ts\ee^{\alpha^{}_{1} + \cdots + \alpha^{}_{d-1} } =
  a^{}_{r}$ holds for all\/ $1\leqslant r\leqslant d \! - \! 1$, with
  the coefficients\/ $a^{}_{r}$ from Eq.~\eqref{eq:A-def}.  \qed
\end{theorem}

One can analyse the situation further from here in various ways. For
instance, one finds
$\det \bigl( \frac{\partial x_i}{\partial \alpha_{\nts j}}\bigr) =
\ee^{- d \ts (\alpha^{}_1 + \cdots + \alpha^{}_{d-1})} $, and one can
determine bounding surface sheets for the parameter region as
before. We leave details to the interested reader.

\begin{remark}
  The above analysis can also be carried out for general (Abelian)
  group{\ts}-based models in the sense of \cite[Sec.~7.3.2]{Steel}.  In
  particular, the embedding problem for the $C_2 \!\times\! C_2$ case,
  otherwise known as the Kimura 3{\ts}ST model \cite{Kimura}, has
  recently been treated in \cite{RL}. The authors demonstrate in
  Example 3.1 of their work that there exist Kimura 3{\ts}ST Markov
  matrices that are \emph{not} embeddable via a generator of this type.
  However, their example is embeddable with a $C_4$ circulant
  generator $Q$. Hence, this situation is in perfect analogy with our
  Example~\ref{ex:no-go} above, and the explanation of the phenomenon
  is the same.  \exend
\end{remark}

\section{Other classes for $d=3$}\label{sec:three}

While we know from Proposition~\ref{prop:necessary} that any
irreducible $M\in E_3$ must be  positive, this is no longer
the case within $\cE_3$, as can be seen from
\[
   \begin{pmatrix} 1-a & a & 0 \\
   b & 1-b & 0 \\ 0 & 0 & 1 \end{pmatrix}
   \begin{pmatrix} 1 & 0 & 0 \\
   0 & 1-c & c \\ 
   0 & d & 1-d \end{pmatrix}
   \, = \, \begin{pmatrix} 
   1-a & a \ts (1-c)& a \ts c \\
   b & (1-b)(1-c) & (1-b) \ts c \\
    0 & d & 1-d \end{pmatrix} .
\]
When $a,b,c,d \in (0,1)$, the two block matrices on the left are
reducible, but embeddable, while the matrix on the right is primitive,
because its square is  positive. However, the matrix itself
still contains a zero element, so cannot be embeddable due to
\mbox{Proposition~\ref{prop:necessary}{\ts}\eqref{part:four}}.  In
particular, even if the right-hand side can be written as $\ee^B$,
with $B$ determined via the BCH formula say, $B$ has vanishing row
sums, but at least one off-diagonal element will be negative.
Consequently, $E_3 \subsetneq \cE_3$, and the same phenomenon occurs
for all $d\geqslant 3$.  Using other canonical ways to embed $E_2$
into $E_3$, one sees that $\cE_3$ contains non-embeddable primitive
matrices with a single $0$ in one off-diagonal entry, for any of the
six possible choices.

Moreover, as can be seen from \cite[Ex.~14]{Davies}, there are
positive Markov matrices that satisfy the determinant condition, but
are not embeddable, which shows the higher complexity for
$d\geqslant 3$. Let us thus look at some subclasses that are related
to various matrix subalgebras of $\mathrm{Mat} (3,\RR)$.  Our results
can alternatively be derived via a normal form for Markov matrices,
which also underlies some of the analysis in \cite{Joh2,Car,CC}.  A
spectral characterisation of embeddable matrices with simple spectrum
is given in \cite[Cor.~1.3]{Joh2}, while those with spectrum
$\{ 1, \lambda, \lambda \}$ and $0<\lambda < 1$ are covered by
\cite[Prop.~1.4]{Joh2}. The case with $\lambda<0$ was later solved in
\cite{Car,CC}. Our approach is based on the algebraic structure of
(semi-)stable properties, which gives alternative insight.

Before we embark on the special cases, let us take a closer look at
the underlying algebraic structure for $d=3$. When the minimal polynomial of a
Markov matrix $M$ has degree $3$, we know that $M$ is cyclic, and we
can investigate $M = \ee^Q$ under the condition that the generator
satisfies $Q\in \alg (A)$ with $A=M-\one$, by
Lemma~\ref{lem:type}. When the degree is $1$, the polynomial must be
$z-1$, which is the trivial case $M=\one=\ee^{\nix}$.  So, we only
have to worry about the case where the minimal polynomial of $M$ is
$(z-1) (z-\lambda)$ for some $\lambda \in (-1,1)$.

\begin{lemma}\label{lem:alg-1D}
  Let\/ $M$ be an embeddable Markov matrix for\/ $d=3$ with minimal
  polynomial of degree\/ $2$. Then, $M$ is diagonalisable, with\/
  $\alg(A) = \alg(R)$ being one-dimensional, where\/ \mbox{$A=M-\one$}
  and\/ $R=\lim_{n\to\infty} M^n-\one$.  Moreover, any generator\/ $Q$
  with\/ $M=\ee^Q$ is diagonalisable, and one of the following cases
  applies:
\begin{enumerate}\itemsep=2pt
\item $\dim(\alg(Q))=1$, so\/ $Q^2 = -q \ts Q \ne \nix$ for some\/
    $q>0$, and\/ $\alg(Q)=\alg(A)=\alg(R)$. If\/ $1$ is
    a simple eigenvalue of\/ $M$, the generator\/ $Q$ must be of
    equal-input type.
\item $\dim(\alg(Q))=2$ and\/ $1$ is a simple eigenvalue 
   of\/ $M$, which implies that\/ $Q$ is simple, 
   with\/ $\sigma(Q) = \{0,\mu^{}_{\pm}\}$, where\/ 
   $\mu^{}_{\pm} = x \pm m \pi\ii$
    for some\/ $m\in\ZZ\setminus \{0\}$ with $x<0$. 
\end{enumerate}    
   In particular, $M$ is of equal-input type whenever\/ $1$ is 
   a simple eigenvalue of\/ $M$. 
\end{lemma}

\begin{proof}
  Assume $M$ as stated, with eigenvalues $1$ and $\lambda$.  This
  means that $A=M-\one$ has minimal polynomial $z (z + c)$ with
  $c = 1-\lambda > 0$, hence eigenvalues $0$ and $-c$, one of them
  with multiplicity $2$. In either case, both $A$ and $M$ are
  diagonalisable.

  Clearly, $A$ satisfies $A^2 = -c A$, and $\dim(\alg(A))=1$ by
  Fact~\ref{fact:one-dim}.  Now, assume that $M=\ee^Q$, with $Q$ a
  generator, and define $R=\lim_{t\to\infty}\ee^{t\ts Q}-\one$ as
  before, where $R^2=-R$ from Corollary~\ref{coro:M-asymp}.  By
  Lemma~\ref{lem:QandA}, we have $R\in\alg (Q) \cap \ts \alg (A)$,
  which implies $R = c^{-1} A$. In particular, $\alg(A) = \alg(R)$ is
  one-dimensional, while diagonalisability of $Q$ follows from that of
  $M$ in conjunction with Fact~\ref{fact:diag}.

  As $Q\ne \nix$, we have $\dim(\alg(Q))\in\{1,2\}$.  But $Q$
  diagonalisable means it is either simple or satisfies
  $Q^2 = -q \ts Q$ for some $q>0$.  In the latter case, since
  $A\in \alg(Q)$, the generators $A$ and $Q$ both are non-trivial
  multiples of $R$, which gives the equalities of the one-dimensional
  algebras.  If $1\in\sigma (M)$ is simple, the equal-input property
  of $Q$, which has $0$ as a simple eigenvalue, follows from
  Corollary~\ref{coro:eq-input-gen}. The equal-input property is then
  inherited by $M$.

  When $Q$ is simple, $Q^2$ and $Q$ are linearly independent, so
  $\dim(\alg(Q))=2$. Then, we have
  $\sigma(M)=\{1,\ee^{\mu^{}_{+}}, \ee^{\mu^{}_{-}}\}$ with
  $\mu^{}_{_\pm} = x \pm \ii y$. In this situation, we must have
  $\ee^{\mu^{}_{+}} = \ee^{\mu^{}_{-}}$, which is only consistent if
  $x<0$, as $\ee^x<1$ by Elving's theorem, and
  $\ee^{\ii y} = \ee^{-\ii y}=\pm 1$.  Since $Q$ is simple, this gives
  $y=m \pi$ with $m\in\ZZ\setminus\{0\}$.

  As $0\in\sigma (A)$ must be simple in the last case, $A$ is an
  equal-input generator by Corollary~\ref{coro:eq-input-gen} again,
  which also completes the argument for the final claim.
\end{proof}

Let us now turn our attention to some special matrix classes for
$d=3$.

\subsection{Symmetric matrices for $d=3$}

The approximation approach fails for symmetric Markov matrices with
negative eigenvalue as all nearby embeddable, symmetric matrices have
some negative eigenvalue with even algebraic multiplicity by
Proposition~\ref{prop:necessary}{\ts}(4). As these matrices are also
diagonalisable, they can never be cyclic. However, as mentioned after
Definition~\ref{def:stable}, being symmetric with non-negative
spectrum has the required approximation property, whence an embeddable,
symmetric Markov matrix $M$ with $\sigma (M) \subset (0,1]$ can always
be written as $\ee^Q$ with $Q$ symmetric, by
Proposition~\ref{prop:extra}. So, let us look at the general symmetric
(and then automatically doubly stochastic) Markov generator
\begin{equation}\label{eq:Q-sym}
   Q \, = \, Q^T \, = \, \begin{pmatrix} 
   -\alpha -\beta & \alpha & \beta \\
   \alpha & -\alpha-\gamma & \gamma \\ 
   \beta & \gamma & -\beta-\gamma \end{pmatrix}
\end{equation}
with $\alpha, \beta, \gamma \geqslant 0$.
One has $\sigma (Q) = \{ 0, {-\varDelta + s}, 
{-\varDelta - s} \} \subset \RR$ with
\begin{equation}\label{eq:del-s}
    \varDelta \, = \, \alpha + \beta + \gamma
    \quad \text{and} \quad
    s \, = \, \sqrt{\alpha^2 + \beta^2 + \gamma^2 -
    \alpha \beta  - \beta \gamma -\gamma \alpha \,} \ts ,
\end{equation}
where
$ \alpha \beta + \beta \gamma + \gamma \alpha \leqslant \alpha^2 +
\beta^2 + \gamma^2$ by Cauchy--Schwarz and
$\alpha^2 + \beta^2 + \gamma^2 \leqslant (\alpha + \beta + \gamma)^2$
due to $\alpha, \beta , \gamma\geqslant 0$, hence
$0 \leqslant s \leqslant \varDelta$, and $\ee^{\sigma(Q)}\subset (0,1]$. 
In particular, any symmetric
generator is negative semi-definite. Moreover, $s=0$ is only possible
for $\alpha = \beta = \gamma$, which means $\varDelta = 3\ts \alpha$
and results in
$\ee^Q = \ee^{-\varDelta J_3} = \one + (1-\ee^{-\varDelta}) J_3$,
which is to be compared with Eq.~\eqref{eq:J-exp}.

The rather simple structure of the eigenvalues follows easily from the
observation that $Q + (\alpha + \beta + \gamma) \one$ is an
anti-circulant matrix; see \cite[p.~156]{PJD} for more on this matrix
class. Also, eigenvectors can be computed in closed form and read
\[
      u^{}_{0} \, = \, \begin{pmatrix} 1 \\ 1 \\ 1 \end{pmatrix}
      \quad \text{and} \quad
      u^{}_{\pm} \, = \, \begin{pmatrix}
      -s^2 + \gamma^2 - \alpha \beta \pm (\alpha + \beta) s \\
      \alpha^2 - \beta \gamma \mp \alpha s \\
      \beta^2 - \alpha \gamma \mp \beta s
      \end{pmatrix} ,
\] 
which are mutually orthogonal. In Dirac notation, let
$|u^{}_{0} \rangle$ and $|u^{}_{\pm}\rangle$ be the normalised
eigenvectors derived from this, which give the three projectors
\begin{equation}\label{eq:proj}
    |u^{}_{0} \rangle \langle u^{}_{0}| \, = \, I^{}_3 
    \, = \, \one + J^{}_3
    \quad \text{and} \quad
    |u^{}_{\pm}\rangle \langle u^{}_{\pm}| \, = \,
    -\myfrac{1}{2} \ts J^{}_{3} \mp \myfrac{\varDelta}{2 s} 
    \ts I^{}_{3} \pm \myfrac{1}{2 s} \ts
    \begin{pmatrix} \gamma & \alpha & \beta \\
    \alpha & \beta & \gamma \\ \beta & \gamma & \alpha
    \end{pmatrix} ,
\end{equation}
with $I_3$ as in Remark~\ref{rem:I-mat}, as can be checked by an
explicit computation.

\begin{lemma}\label{lem:sym}
  If\/ $Q$ is the generator from Eq.~\eqref{eq:Q-sym}, its exponential
  is given by
\[
     \ee^Q \, = \, \Bigl( 1 -  \frac{\sinh (s)}{s} \ts
     \varDelta \ts \ee^{-\varDelta}\Bigr) I^{}_3 -  \cosh (s) \ts
     \ee^{-\varDelta} \ts J^{}_3 + \frac{\sinh (s)}{s} \ts\ts
     \ee^{-\varDelta} 
     \begin{pmatrix} \gamma & \alpha & \beta \\
    \alpha & \beta & \gamma \\ \beta & \gamma & \alpha
    \end{pmatrix} ,
\]   
   which correctly covers the limiting case\/ 
   $s \, \text{\footnotesize $\searrow$} \, 0$, where\/
   $\alpha = \beta = \gamma$ and\/
   $\ts \ee^Q = \one + \bigl(1 - \ee^{-\varDelta} \bigr)\ts J_3$.
\end{lemma}

\begin{proof}
  Given $Q$, which is diagonalisable, the spectrum of $\ee^Q$ is
  $\bigl\{ 1, \ee^{-\varDelta + s}, \ee^{-\varDelta - s}
  \bigr\}$. Employing the projectors from Eq.~\eqref{eq:proj}, one
  then obtains
\[
     \ee^Q \, = \,  |u^{}_{0} \rangle \langle u^{}_{0}| \, + \,
     |u^{}_{+}\rangle \ee^{-\varDelta + s}\langle u^{}_{+}|  \, + \,
     |u^{}_{-}\rangle \ee^{-\varDelta - s}\langle u^{}_{-}| \ts ,
\]    
which leads to the formula by a simple calculation. The claim on the
limit follows from $\lim_{s\to 0} \frac{\sinh (s)}{s} = 1$ and the
fact that $s=0$ means $\alpha = \beta = \gamma$ and
$\varDelta = 3 \ts \alpha$.
\end{proof}

\begin{remark}\label{rem:not-Lie}
  The symmetric matrices do not form a matrix algebra, because
  $(AB)^T = B^T \! A^T$ and $\mathrm{Mat} (3,\RR)$ contains symmetric
  matrices that do not commute. Likewise, symmetric generators do not
  form a Lie algebra, because $[A,B]^T = - \bigl[A^T, B^T \ts \bigr]$.
  Consequently, the corresponding symmetric Markov matrices of the
  form $\ee^Q$ are not closed under matrix multiplication, as
  discusses in \cite{Jeremy}. However, symmetric matrices \emph{do}
  form a Jordan algebra, with
  $\{ A,B \}^T = \{ A, B\} = \frac{1}{2} (AB + BA)$. This observation
  turned out to be a helpful property for identifying (semi-)stable
  matrix classes.  \exend
\end{remark}

Let us now consider a general symmetric Markov matrix, written as
\begin{equation}\label{eq:M-sym}
    M \, = \, \begin{pmatrix} 1-a-b & a & b \\ a & 1-a-c & c \\
    b & c & 1-b-c \end{pmatrix} 
\end{equation}
with $a,b,c\in[0,1]$ as well as $a+b, a+c, b+c \in [0,1]$.  Note that
any such matrix is also doubly stochastic, and that
$M + (a+b+c-1) \one$ is again anti-circulant.

Clearly, we can have $a=b=c=0$, which is $\one = \ee^{\nix}$.  More
generally, let us assume that $M$ is embeddable. If $a=0$,
Proposition~\ref{prop:necessary}{\ts}\eqref{part:four} implies that
$b\ts c=0$, and likewise for $b=0$ or $c=0$. So, we have either 
$a\ts b\ts c>0$ or $a\ts b + b\ts c + c\ts a = 0$.  If $a=1$, we must
have $b=0$ and hence also $c=0$, which then fails to be embeddable 
by Theorem~\ref{thm:Kendall}, and analogously for $b=1$ or $c=1$. In
fact, whenever two of the numbers are zero, Theorem~\ref{thm:Kendall}
implies that the third lies in the interval
$\bigl[ 0, \frac{1}{2} \bigr)$, which gives us the following result,
where the second claim can be seen constructively from another
application of Eq.~\eqref{eq:log-calc}.

\begin{fact}\label{fact:sym-zero}
  When a symmetric Markov matrix of the form\/ \eqref{eq:M-sym} fails
  to be positive, it is embeddable if and only if\/
  $a\ts b + b\ts c + c\ts a =0$ together with\/
  $0 \leqslant \max (a,b,c) < \frac{1}{2}$.  \qed
\end{fact}

The general case can be stated as follows.

\begin{theorem}
  Let\/ $M \in \cM_3$ be a symmetric Markov matrix as given in
  \eqref{eq:M-sym}, with\/ parameters\/ $a,b,c\geqslant 0$ and\/
  $\max (a+b,a+c,b+c) \leqslant 1$. Then, the following statements are
  equivalent.
\begin{enumerate}\itemsep=2pt
\item $M$ is embeddable and has positive spectrum.
\item $M$ is embeddable with a symmetric generator.
\item There are non-negative numbers\/ $\alpha, \beta, \gamma$
   such that
\[
   \begin{pmatrix} a \\ b \\ c \end{pmatrix} \, = \,
   \myfrac{1}{3} \Bigl( 1 -  \frac{\sinh (s)}{s} \ts \varDelta \ts\ts
   \ee^{-\varDelta} - \cosh(s) \ts\ts \ee^{-\varDelta} \Bigr)
   \begin{pmatrix} 1 \\ 1 \\ 1 \end{pmatrix}  +
   \frac{\sinh (s)}{s} \ts\ts \ee^{-\varDelta}
   \begin{pmatrix} \alpha \\ \beta \\ \gamma
   \end{pmatrix},
\]   
    with\/ $\varDelta$ and\/ $s$ as in Eq.~\eqref{eq:del-s}.
\end{enumerate}
Futhermore, there are embeddable, symmetric Markov matrices with
negative eigenvalues, but none of them can be written as\/
$\ee^Q$ with\/ $Q$ symmetric.
\end{theorem}

\begin{proof}
  (1) $\Rightarrow$ (2) follows from the argument used in the proof of
  Proposition~\ref{prop:extra}, because an embeddable, symmetric
  Markov matrix with positive spectrum can be approximated by cyclic,
  embeddable ones, while (2) $\Rightarrow$ (1) is clear. Now,
  Lemma~\ref{lem:sym} gives the general form of $\ee^Q$ with $Q$ a
  symmetric generator, and (3) is the resulting condition on the
  parameters.

  The final claim follows from the matrix $M$ of Example~\ref{ex:no-go}
  together with the fact that $\sigma (\ee^Q) = \ee^{\sigma(Q)} \subset
  (0,1]$ for any symmetric generator $Q$.
\end{proof}

Let us comment a little on the situation at hand.  If $M$ from
Eq.~\eqref{eq:M-sym} is embeddable with a symmetric rate
matrix, so $M=\ee^Q$ with $Q$ as in
Eq.~\eqref{eq:Q-sym}, one has
\begin{equation}\label{eq:M-spec}
   \sigma (M) \, = \, \bigl\{ 1, \ee^{-\varDelta+s},
     \ee^{-\varDelta-s} \bigr\} \, \subset \, ( 0, 1] 
\end{equation}
with $\varDelta$ and $s$ as in Eq.~\eqref{eq:del-s}.  We can now
compare the coefficients of the characteristic polynomial of $M$ with
the corresponding expressions from Vieta's formula. This gives three
necessary conditions for embeddability as follows.  First, in line with
Proposition~\ref{prop:necessary}{\ts}(2), one has
$\det (M) \in (0,1]$, which means
\[
    3 (ab + ac + bc) \, \leqslant \, 2 (a+b+c) \, < \,
    1 + 3 (ab + ac + bc) \ts .
\]
Next, one finds $\tr (M) \in (1,3]$, which is equivalent with
\[
    0 \, \leqslant \,  a+b+c  \, < \, 1 \ts ,
\]
while the third condition on the eigenvalues can be stated as
\[
     0 \, \leqslant \, 3 \ts (ab + bc + ca) \, = \, 
    \bigl( 1 - \ee^{-\varDelta + s} \bigr)
      \bigl( 1 - \ee^{-\varDelta - s} \bigr)
      \, < \, 1 \ts .
\]  
All three conditions are sharp in the sense that each possible value
can be realised. Note that these conditions, even when taken together,
are not sufficient for the embeddability of $M$.

Two other properties follow from Eq.~\eqref{eq:M-spec}, namely that
$M$ is positive definite and that $M-\one$ has spectral radius less
than $1$.  The former, via Sylvester's criterion, means $\det (M)>0$
together with $\max (a+b,a+c,b+c) <1$, which also follows from the
trace condition, and
\[
    1 + (ab+bc+ca) \, > \, (a+b+c) + \max (a,b,c) \ts ,
\]
which, in this case, is not an independent condition either.  With
$A \defeq M-\one$, the other property implies that the symmetric
matrix
\[
    Q' \, \defeq \sum_{m\geqslant 1} \frac{(-1)^{m-1}}{m} A^m
    \, = \, \log (M) \, \in \, \alg (A) \, \subset \, \cA^{}_{0}
\]  
is well defined, has vanishing row sums because $A$ is a generator,
and satisfies $M = \ee^{Q'}$. The difficult step, when starting from
this formula, consists in formulating a condition on $a,b,c$ that
ensures the off-diagonal elements of $Q'$ to be non-negative, that is,
its Metzler property. No criterion in polynomial form, say, exists for
this.

As we saw in Example~\ref{ex:no-go}, there are symmetric Markov
matrices with negative eigenvalues that are still embeddable, but never
with a symmetric generator. These cases are naturally included in the
class of doubly stochastic matrices, which we consider below in
Section~\ref{sec:doubly-stoch}.

\subsection{Constant input matrices}\label{sec:c-input}

Let us go back to the constant-input Markov matrices from
Section~\ref{sec:equal-input} and Remark~\ref{rem:constant-input}.
Since $d=3$ is odd, Theorem~\ref{thm:equal-input} and its constructive
proof imply that a constant-input Markov matrix $M$ is embeddable with
a constant-input generator if and only if the summatory parameter $c$
of $M$ satisfies $0\leqslant c < 1$, where $c=0$ means $M=\one$. For
$1< c \leqslant 2$, we have further embeddable cases, but then
necessarily with doubly stochastic generators that are \emph{not} of
equal-input type; compare Example~\ref{ex:no-go} and
Fact~\ref{fact:J-comm}.

So, assume $M=\ee^Q$ is constant-input with $c^{}_{M}=c>1$ and
with $Q$ being doubly stochastic, the most general form of which is 
\begin{equation}\label{eq:d-stoch-gen}
    Q \, = \, \begin{pmatrix}  
    -\alpha - \beta & \alpha + \epsilon & \beta - \epsilon \\
    \alpha - \epsilon & -\alpha - \gamma & 
        \gamma + \epsilon \\ \beta + \epsilon & 
     \gamma - \epsilon & - \beta - \gamma \end{pmatrix}
\end{equation}
with $\alpha, \beta, \gamma \geqslant 0$ and
$\lvert \epsilon \rvert \leqslant \min (\alpha, \beta, \gamma )$. 
One finds $\sigma (Q) = \{ 0, -\varDelta + s^{}_{\nts \epsilon} , 
-\varDelta - s^{}_{\nts \epsilon} \}$ with 
\begin{equation}\label{eq:dsq-para}
  \varDelta = \alpha + \beta + \gamma
  \quad \text{and} \quad
  s^{}_{\nts \epsilon} = 
  \sqrt{\alpha^2 + \beta^2 + \gamma^2 - \alpha \beta -
  \beta \gamma - \gamma \alpha - 3\ts \epsilon^2 \ts }.
\end{equation}  
A matching set of eigenvectors can be given as
\begin{equation}\label{eq:eig}
      v^{}_{0} \, = \, \begin{pmatrix} 1 \\ 1 \\ 1 \end{pmatrix}
      \quad \text{and} \quad
      v^{}_{\pm} \, = \, \begin{pmatrix}
      -\alpha + \gamma -\epsilon \mp s^{}_{\nts \epsilon} \\
      \alpha - \beta-\epsilon \pm s^{}_{\nts \epsilon} \\
       \beta - \gamma + 2 \ts s^{}_{\nts \epsilon}
      \end{pmatrix} ,
\end{equation}
where $v^{}_{0}$ is perpendicular to $v^{}_{\pm}$.

Since we now have
$\sigma (\ee^Q) = \{ 1, 1-c, 1-c\} = \{ 1, \ee^{-\varDelta +
  s_{\epsilon}}, \ee^{-\varDelta - s_{\epsilon}} \}$ with $\varDelta$
and $s^{}_{\nts \epsilon}$ from \eqref{eq:dsq-para}, we see that
$\ee^{-\varDelta \pm s_{\epsilon}}$ must be \emph{negative}, which implies
that $s^{}_{\epsilon} = (2k+1) \pi \ii$ for some $k\in\ZZ$ and results
in $c = 1 + \ee^{-\varDelta}$. Since
$\alpha^2 + \beta^2 +\gamma^2 \geqslant \alpha\beta + \beta\gamma +
\gamma\alpha$, we have
\[
    -3 \ts \epsilon^2 \, \leqslant \,  s^2_{\epsilon}
    \, = \, - (2 k + 1)^2 \pi^2
\]
for any chosen $k\in\ZZ$, which implies
$ \lvert \epsilon \rvert \ts \sqrt{3} \geqslant \lvert 2 k+1\rvert
\,\pi$.  To make sure that $Q$ is a generator, we also need
$\lvert \epsilon \rvert \leqslant \min (\alpha,\beta,\gamma)$, hence
$\varDelta \geqslant 3 \ts \lvert \epsilon \rvert$. Together, for any
chosen $k\in\ZZ$, this means
$c \leqslant 1 + \ee^{- \lvert 2k+1\rvert \ts \pi \sqrt{3}}$, which
takes the form $c \leqslant 1 + \ee^{-\pi \sqrt{3}}$ for $k=0$ or
$k=-1$. This upper bound is the value we saw in
Example~\ref{ex:no-go}.

\begin{coro}\label{coro:exceptional}
  Let\/ $M$ be a constant-input Markov matrix for\/ $d=3$ with
  corresponding summatory parameter\/ $c^{}_{\nts M}>1$. Then, $M$ is
  embeddable if and only if\/ $M=\ee^Q$, with\/ $Q$ a doubly
  stochastic generator. This is possible if and only if\/
  $1 < c^{}_{\nts M} \leqslant 1+\ee^{-\pi\sqrt{3}}$.
\end{coro}

\begin{proof}
  We can proceed constructively with the choice
  $\alpha = \beta = \gamma$.  Choose $Q_{\alpha} = 3\ts \alpha J + T$
  with the generator $J=J_3$ from Example~\ref{ex:no-go}, which is the
  unique constant-input generator with $J^2 = -J$, and the matrix
\[
     T \, = \, \frac{\pi}{\sqrt{3}}
     \begin{pmatrix} 0 & 1 & -1 \\ -1 & 0 & 1 \\
     1 & -1 & 0 \end{pmatrix} ,
\]   
which commutes with $J$ and satisfies $\ee^T = \one+2 J$, by a
calculation similar to the one used in Example~\ref{ex:no-go}.  Note
that $Q_{\alpha}$ is a generator precisely when
$\alpha\geqslant \pi/\sqrt{3}$. One finds
\[
     M \, = \, \ee^{Q_{\alpha}} \, = \, \bigl( \one + 
     (1 - \ee^{- 3 \alpha}) J \bigr) \bigl( \one + 2 J \bigr)
     \, = \, \one + (1 + \ee^{-3 \alpha}) J \ts ,
\] 
which is a constant-input Markov matrix with
$c^{}_{\nts M} = 1+\ee^{-3 \alpha}$.  With the admissible choices for
$\alpha$, one exhausts the claimed range of the parameter
$c^{}_{\nts M}$.
\end{proof}

Note that some of the arguments in the last proof are related to the 
claims from Lemma~\ref{lem:alg-1D}, because $M-\one$ in 
Corollary~\ref{coro:exceptional} is a constant-input generator.

\begin{remark}\label{rem:multi}
  We have used the choice $k=0$ to get the maximal range.
  Alternatively, when using $k\in\NN$, so
  $\epsilon^{}_{k} = (2k+1) \pi/\sqrt{3}$ and
  $\alpha\geqslant \epsilon^{}_{k}$, one obtains the range
  $1 < c^{}_{\nts M} \leqslant 1 + \ee^{-(2k+1)\pi\sqrt{3}} $ with
  another embedding solution. So, for smaller and smaller regions, one
  gets an increasing number of solutions to the embedding problem; see
  \cite{Len} for a related discussion.

  In fact, due to $(\one + 2 J)^2 = \one$, one has
  $\ee^{nT} = \one + \bigl( (-1)^{n+1} + 1 \bigr) J$, and this gives
  extra embedding solutions also for $c<1$. Indeed, when
  $n=2m\geqslant 0$ is even, $Q_{\alpha,m} = 3 \ts \alpha J + 2 m T$
  is a generator when $\alpha \geqslant 2 m \pi/\sqrt{3}$, and
\[
     M\, = \,\ee^{Q_{\alpha,m}} \, = \, \one + (1-\ee^{-3\alpha}) J
\]
then is a constant-input Markov matrix with parameter
$c^{}_{\nts M}=1-\ee^{-3\alpha}$.  Note that this produces an
analogous multi-embedding phenomenon as in the previous case, but this
time for the range
$1-\ee^{-2m\pi\sqrt{3}} \leqslant c^{}_{\nts M} < 1$.  These extra
solutions (with $m\in\NN$) are doubly stochastic generators that are
not of equal-input type.  \exend
\end{remark}

\subsection{Doubly stochastic matrices}\label{sec:doubly-stoch}

A natural extension of symmetric Markov matrices is provided by the
family of doubly stochastic ones. The latter form a closed convex set
that is again a monoid. The extremal elements are the $d\ts !$
permutation matrices, compare \cite{Joh2}, which are linearly
dependent for $d\geqslant 3$.

There are $6$ extremal matrices for $d=3$, but the corresponding
monoid is four-dimensional. Indeed, a simple calculation, compare
Fact~\ref{fact:J-comm}, shows that any doubly stochastic matrix can be
parametrised as
\begin{equation}\label{eq:ds-def}
    M \, = \, \begin{pmatrix}  1-a-b & a+e & b-e \\
    a-e & 1-a-c & c+e \\ b+e & 
    c-e & 1 - b - c \end{pmatrix}
\end{equation}
with $a,b,c \in [0,1]$ and $e\in [-1,1]$, subject to the obvious
constraints to make $M$ a Markov matrix, namely
$\lvert e \rvert \leqslant \min (a,b,c)$ and
$\max (a+b,a+c,b+c) \leqslant 1$.  The locally constant, topological
dimension clearly is $4$, and there is only one additional parameter
in comparison to Eq.~\eqref{eq:M-sym}, namely $e$ (for \emph{excess}),
with $e=0$ giving the symmetric matrices.  Note that $M$ is normal if
and only if $e=0$. Here, $M$ has spectrum
$\sigma (M) = \{ 1, 1-\varDelta^{}_M + s^{}_{\nts M}, 1 -
\varDelta^{}_M - s^{}_{\nts M} \}$ with $\varDelta^{}_M = a+b+c$ and
$s^{}_{\nts M} = \sqrt{a^2 + b^2 + c^2 - ab - bc - ca - 3 e^2}$.

The matrix $M$ from Example~\ref{ex:no-go} is doubly stochastic and
shows that some subtle phenomena can occur. 
However, if a doubly stochastic Markov matrix is embeddable, so $M=\ee^Q$
for some generator $Q$, the row vector $(1, \ldots , 1)$ is a left eigenvector
of both $M$ and $Q$, with eigenvalue $0$ for $Q$ by the spectral mapping 
theorem because $Q$ is a rate matrix; compare
Proposition~\ref{prop:G-asymp}{\ts}(1).  This implies that the 
generator $Q$ is doubly stochastic.

\begin{coro}\label{coro:d-stoch}
  A doubly stochastic Markov matrix\/ $M$ is embeddable if and 
  only if\/ $M=\ee^Q$ with\/ $Q$ a doubly stochastic
  generator.  \qed
\end{coro}

Note that the generator $Q$ of Eq.~\eqref{eq:d-stoch-gen} is of
equal-input type if and only if $\epsilon=0$ together with
$\alpha = \beta = \gamma$, which means that it is then a
constant-input generator.

\begin{theorem}
  Let\/ $M$ be the general doubly stochastic Markov matrix in\/
  $\cM_3$, as given in Eq.~\eqref{eq:ds-def}, 
  with\/ parameters\/ $a,b,c\geqslant 0$ and\/
  $e\in \RR$ subject to\/ $\max (a+b, a+c, b+c) \leqslant 1$ and\/
  $\lvert e \rvert \leqslant \min (a,b,c)$. Let\/ $p$ denote the
  minimal polynomial of\/ $M$. Then, $M$ is embeddable if and only if
  one of the following situations applies.
\begin{enumerate}\itemsep=2pt
\item $\degr (p)=1$, which means\/ $p(z) = (z-1)$, and thus\/
  $M=\one = \ee^{\nix}$;
\item $\degr (p) = 2$, which implies\/ $p(z) = (z-1)(z-\lambda)$ for
  some\/ $\lambda \in (-1,1)$, so $M$ is diagonalisable; if\/ $1$ has
  multiplicity\/ $2$, we have\/ $e=\epsilon=0$, $M$ is symmetric,
  and\/ $ab + bc + ca=0$ together with\/
  $0 < \max (a,b,c) < \frac{1}{2}$; if\/ $1$ is simple, $A=M-\one$
  must be a constant-input generator with\/ $a=b=c$, parameter sum\/
  $c^{}_{\nts M}= 3 a =1\pm \ee^{-\varDelta}$, and\/
  $0 < c^{}_{\nts M} \leqslant 1+\ee^{-\pi\sqrt{3}} $ with\/
  $c^{}_{\nts M} \ne 1$;
\item $\degr (p) = 3$, and there are non-negative numbers\/
  $\alpha, \beta, \gamma$, and some\/ $\epsilon\in\RR$, such that
\[
   \begin{pmatrix} a \\ b \\ c \\ e \end{pmatrix} \, = \,
   \myfrac{1}{3} \Bigl( 1 -  \frac{\sinh (s^{}_{\epsilon})}
          {s^{}_{\epsilon}} \ts \varDelta \ts\ts
   \ee^{-\varDelta} - \cosh(s^{}_{\epsilon}) \ts\ts \ee^{-\varDelta} \Bigr)
   \begin{pmatrix} 1 \\ 1 \\ 1 \\ 0 \end{pmatrix}  +
   \frac{\sinh (s^{}_{\epsilon})}{s^{}_{\epsilon}} \ts\ts \ee^{-\varDelta}
   \begin{pmatrix} \alpha \\ \beta \\ \gamma \\ \epsilon
   \end{pmatrix},
\]   
with\/ $\varDelta$ and\/ $s^{}_{\epsilon}$ as in
Eq.~\eqref{eq:dsq-para}, and\/
$\lvert \epsilon\rvert \leqslant \min (\alpha, \beta, \gamma)$.
\end{enumerate}    
\end{theorem}

\begin{proof}
  Case (1) is trivial, while the case distinction in (2) follows from
  a simple calculation with the eigenvalues of $M$ and its
  consequences for $\varDelta$ and $s^{}_{\epsilon}$. When
  $\varDelta=s^{}_{\epsilon}$, we are back to
  Fact~\ref{fact:sym-zero}, while $s^{}_{\epsilon}=0$ implies $M$ to
  be equal-input and doubly stochastic, hence constant-input, and the
  condition stated here follows from Theorem~\ref{thm:equal-input}
  and Corollary~\ref{coro:exceptional}.
   
  For Case (3), $M$ is either simple or has a real eigenvalue
  $\lambda$, with $\lvert \lambda \rvert < 1$ and algebraic
  multiplicity $2$, but geometric multiplicity $1$, and hence a
  non-trivial Jordan block in the Jordan normal form. When $M$ is
  simple, hence cyclic, we are in Case (2) of Lemma~\ref{lem:alg-1D},
  so $Q$ is diagonalisable as well and
  $A=M-\one \in \alg (Q) = \langle Q, J \ts \rangle^{}_{\RR}$ with
  $J=J_3$ as before. Here, $A,Q$ and $J$ are simultaneously
  diagonalisable, with a matrix that derives from the eigenvectors of
  $Q$ as given in \eqref{eq:eig}.  Now, we can use $A= u\ts Q + v J $
  and compare eigenvalues, which results in the equation as
  stated. The further constraints guarantee the generator property of
  $Q$.
   
  Finally, the remaining Jordan case is obtained as a limit of such
  simple matrices, which still gives the same equation for the
  parameters.
\end{proof}

\section{Outlook}\label{sec:outlook}

There are many aspects of the embedding problem that we have not
treated or addressed here, though some were briefly mentioned in our
remarks.  Among them are more general results on uniqueness or
multiple solutions, which becomes increasingly difficult with growing
dimension, or the classification of matrix classes that are connected
with Jordan or Lie algebras, because their number also increases
quickly. A more complete picture should still be achievable up to
$d=4$, while further constraints would be needed beyond.

From the viewpoint of biological application, it seems desirable to
concretely consider matrix classes for $d=4$ that cover the standard
mutation schemes of molecular evolution, which are commonly used in
bioinformatics and in population genetics. Since the number of
relevant matrix classes is much larger than for $d=3$, this needs a
separate treatment. In this context, it would also be relevant to know 
the relation to inhomogeneous Markov chains, as this can cover
time-dependent processes more realistically.

Another direction is the extension of the analysis to countable state
Markov chains, which would require new methods and tools from
functional analysis, or to sub-stochastic matrices and their
generators, which show up increasingly in theoretic and applied
probability. Here, the non-negativity conditions remain the same, but
the row sums for sub-stochastic matrices or generators are either
$\leqslant 1$ or $\leqslant 0$, respectively. One would expect
inequalities that parallel our above results, but little has been
done in this direction so far.

\section*{Acknowledgements}

It is our pleasure to thank F.\ Alberti, B.\ Gardner, P.D.\ Jarvis,
H.\ K\"{o}sters, A.\ Radl and M.~Steel for discussions.  We also thank
the organisers and participants of MAM{\ts}10 in Hobart, Tasmania, for
useful hints on the problem, and two referees for their thoughtful
comments that helped to improve the presentation.  This work was
supported by the German Research Foundation (DFG), within the SPP
1590, and by the Australian Research Council (ARC), via Discovery
Project DP 180{\ts}102{\ts}215.

\end{document}